\documentclass[a4paper,12pt]{article}

\usepackage{amsfonts,graphicx,amsmath,amssymb,amsthm,color,nicefrac,enumerate, layout, bbm, listings,  hyperref}
\usepackage[latin1]{inputenc}
\usepackage[numbers]{natbib}
\usepackage{bbm}

\usepackage{geometry}
\newtheorem{theorem}{Theorem}[section]
\newtheorem{prop}[theorem]{Proposition}
\newtheorem{lemma}[theorem]{Lemma}
\newtheorem{cor}[theorem]{Corollary}
\newtheorem{sett}[theorem]{Setting}

\newcommand{\eps}{\varepsilon}
\newcommand{\R}{\mathbb{R}}

\newcommand{\N}{\mathbb{N}}
\newcommand{\E}{\mathbb{E}}
\renewcommand{\P}{\mathbb{P}}

\newcommand{\1}{\mathbbm{1}}
\newcommand{\HS}{\operatorname{HS}}
\newcommand{\constFun}{\phi}

\newcommand{\Id}{\mathrm{Id}}
\newcommand{\F}[0]{\ensuremath{\mathcal{F}}}

\title{A stochastic Gronwall inequality
and\\
applications to moments,
strong completeness,\\
strong local Lipschitz continuity,
and perturbations}
\author{Anselm Hudde$^1$, Martin Hutzenthaler$^1$,
and Sara Mazzonetto$^{2}$ 
\bigskip
\\
\small{$^1$Faculty of Mathematics, University of Duisburg-Essen,
Germany}
\\
\small{$^2$Institute of Mathematics, University of Potsdam, 
Germany}
\smallskip
}

\begin{document}

\maketitle
\makeatletter
\let\@makefnmark\relax
\let\@thefnmark\relax
\@footnotetext{\emph{AMS 2010 subject classification:} 60H10; secondary: 60E15}
\@footnotetext{\emph{Key words and phrases:} Stochastic Gronwall inequality, stochastic Gronwall lemma, martingale inequality, exponential moments, strong completeness, strong local Lipschitz continuity, perturbation theory}
\makeatother

\begin{abstract}
  There are numerous applications of the classical
  (deterministic)
  Gronwall inequality.
  Recently, Michael Scheutzow discovered
  a stochastic Gronwall inequality
  which provides upper bounds for
  $p$-th moments, $p\in(0,1)$,
  of the supremum of nonnegative scalar continuous processes
  which satisfy a linear integral inequality.
  In this article we complement this with
  upper bounds for $p$-th moments, $p\in[2,\infty)$,
  of the supremum of general It\^o processes
  which satisfy a suitable one-sided affine-linear growth condition.
  As example applications, we improve known results
  on strong local Lipschitz continuity in the starting point of solutions
  of stochastic differential equations (SDEs),
  on (exponential) moment estimates for SDEs,
  on strong completeness of SDEs,
  and
  on perturbation estimates for SDEs.
\end{abstract}

\tableofcontents


\section{Introduction}
  Recently, Scheutzow \cite{Scheutzow2013}
  discovered
  a powerful stochastic Gronwall inequality.
More precisely, Scheutzow~\cite[Theorem 4]{Scheutzow2013}
proves that if
$Z,\alpha,\eta\colon[0,\infty)\times\Omega\to[0,\infty)$
are adapted processes on a filtered probability space $(\Omega,\mathcal{F},\P,(\mathbb{F}_t)_{t\in[0,\infty)})$
with continuous sample paths
and if $M\colon[0,\infty)\times\Omega\to\R$ is a continuous local martingale with $M_0=0$
which satisfy
that $\P$-a.s.\ it holds for all $t\in[0,\infty)$ that
\begin{equation}  \begin{split}\label{eq:Z.intro}
  Z_t\leq \int_0^t\alpha_s Z_s\,ds+M_t+\eta_t,
\end{split}     \end{equation}
then for all $t\in[0,\infty)$, $q_1,q_3\in(0,1)$, $q_2\in(0,\infty)$ with
$\tfrac{1}{q_1}=\tfrac{1}{q_2}+\tfrac{1}{q_3}$ it holds that
\begin{equation}  \begin{split}
     &\left\|\sup_{s\in[0,t]}Z_s\right\|_{L^{q_1}(\P;\R)}
     \leq
      \Big(\min\{4,\tfrac{1}{q_3}\} \tfrac{\pi q_3}{\sin(\pi q_3)}+1
      \Big)^{\frac{1}{q_3}}
     \bigg\|
     \exp\Big(\smallint_0^{ {t}} \alpha_u\,du \Big)
     \bigg\|_{L^{q_2}(\P;\R)}
     \left\|\sup_{s\in[0,t]}\eta_s \right\|_{L^{q_3}(\P;\R)}.
\end{split}     \end{equation}
In this article we complement
Scheutzow~\cite[Theorem 4]{Scheutzow2013}
with upper bounds for $p$-th moments, $p\in[2,\infty)$,
of general mulit-dimensional It\^o processes which satisfy
the \emph{one-sided affine-linear growth condition}~\eqref{eq:lin.growth.intro.thm}
below. Our contribution is as follows:
\begin{itemize}
  \item Marginal estimate~\eqref{eq:c:moments.hilbert.m.intro}:
      We observe that the proof of Hutzenthaler and Jentzen~\cite[Theorem 2.10]{HutzenthalerJentzen2014} transfers to the more general setting of Theorem~\ref{thm:moments:hilbert.m.intro}.
 The estimate resulting from this approach turns out to be suboptimal
   if $\beta$ in~\eqref{eq:lin.growth.intro.thm} is non-zero due to
   an estimate with Young's inequality.
   We show how to avoid the use of Young's inequality 
   by applying the Gronwall-Bellman-Opial inequality in Lemma~\ref{l:nonlinear.Gronwall}
   below.
  \item Uniform estimate~\eqref{eq:c:moments.hilbert.m.sup.intro}:
  We observe
  that 
   It\^o's formula applied 
    for fixed $p\in[2,\infty)$
   to
$(\|X_t\|_{\R^d}^p)_{t\in[0,\infty)}$
  together with the
{one-sided affine-linear growth condition}~\eqref{eq:lin.growth.intro.thm}
and Young's inequality results in inequality~\eqref{eq:Z.intro}
with $Z= (\|X_t\|_{\R^d}^p)_{t\in[0,\infty)}$;
see the proof of
Theorem~\ref{thm:moments:hilbert.m} below for details.
Thus
Theorem 4 in Scheutzow~\cite{Scheutzow2013}
   can be applied to $(\|X_t\|_{\R^d}^p)_{t\in[0,\infty)}$.
 The estimate resulting from this approach turns out to be suboptimal
   if $\beta$ in~\eqref{eq:lin.growth.intro.thm} is non-zero.
   We show how to avoid the use of Young's inequality and how
   to get a better upper bound.
 \item
   We observe that 
the {one-sided affine-linear growth condition}~\eqref{eq:lin.growth.intro.thm}
(or its analog~\eqref{eq:ass.powers2} for more general Lyapunov-type functions)
is satisfied in many applications; see Section~\ref{sec:applications}
for a few examples.
\end{itemize}
\begin{theorem}[A stochastic Gronwall inequality for multi-dimensional It\^o processes]\label{thm:moments:hilbert.m.intro}
Let
$d,m\in\N$,
$T \in (0, \infty)$,
$p\in[2,\infty)$,
let $\|\cdot\|_{\R^d}$, $\|\cdot\|_{\R^m}$, and $\langle \cdot,\cdot\rangle_{\R^d}$ denote Euclidean norm respectively scalar product,
let $\|\cdot\|_{\R^{d\times m}}\colon \R^{d\times m}\to[0,\infty)$
satisfy for all $A=(A_{i,j})_{i\in\{1,\ldots,d\},j\in\{1,\ldots,m\}}\in\R^{d\times m}$ that $\|A\|_{\R^{d\times m}}^2=\sum_{i=1}^d\sum_{j=1}^m|A_{i,j}|^2$,
let $(\Omega, \F, \P)$ be a probability space with a normal filtration $(\mathbb{F}_{t})_{t\in [0,T]}$,
let $W\colon[0,T]\times\Omega\to\R^m$ be a
standard Brownian motion,
let $X,a\colon[0,T]\times\Omega\to \R^d$,
$b\colon[0,T]\times\Omega\to
\R^{d\times m}$,
$\alpha,\beta \colon [0,T] \times \Omega \to [0,\infty)$
be $\mathcal{B}([0,T])\otimes\mathcal{F}$-measurable and adapted stochastic processes
which satisfy $\P$-a.s.\ that
$\int_0^{T}\|a_s\|_{\R^d}+\|b_s\|_{\R^{d\times m}}^2+|\alpha_s|\,ds<\infty$,
which satisfy that $X$ has continuous sample paths,
which satisfy for all $t\in[0,T]$
that it holds $\P$-a.s.\ that $X_t=X_0+\int_0^t a_s\,ds +\int_0^tb_s\,dW_s$,
and
which satisfy that $\P$-a.s.\ it holds for Lebesgue-almost all $t\in[0,{T}]$ that
\begin{equation}  \begin{split}\label{eq:lin.growth.intro.thm}
  \langle X_t,a_t\rangle_{\R^d}+\tfrac{1}{2}\|b_t\|_{\R^{d\times m}}^2
  +\tfrac{p-2}{2}\tfrac{\|\langle X_t,b_t\rangle_{\R^d}\|_{\R^m}^2}{\|X_t\|_{\R^d}^2}
  \leq \alpha_t\|X_t\|_{\R^d}^2+\tfrac{1}{2}|\beta_t|^2.
\end{split}     \end{equation}
 Then
   \begin{enumerate}[(i)]
     \item  
   it holds for all $q_1,q_2\in(0,\infty]$, $t\in[0,T]$
   with $\tfrac{1}{q_1}=\tfrac{1}{q_2}+\tfrac{1}{p}$ that
   \begin{equation}  \begin{split}\label{eq:c:moments.hilbert.m.intro}
     &\|X_{t}\|_{L^{q_1}(\P;\R^d)}
     \leq
     \bigg\|
     \exp\Big(\smallint_0^{ {t}} \alpha_u\,du \Big)
     \bigg\|_{L^{q_2}(\P;\R)}
     \left(
     \left\|
     X_0
     \right\|_{L^p(\P;\R^d)}^2
     +\int_0^t
     \Big\|
     \tfrac{ \beta_s }
     {\exp\left(\smallint_0^s \alpha_u\,du \right)}
     \Big\|_{L^{p}(\P;\R)}^2\,ds
     \right)^{\!\frac{1}{2}}
   \end{split}     \end{equation}
   and
   \item
   it holds for all $q_1,q_2,q_3\in(0,\infty]$
   with $q_3< p$ and $\tfrac{1}{q_1}=\tfrac{1}{q_2}+\tfrac{1}{q_3}$ that
   \begin{align}\label{eq:c:moments.hilbert.m.sup.intro}
     &\left\|\sup_{s\in[0,T]}\|X_s\|_{\R^d}\right\|_{L^{q_1}(\P;\R)}
     \leq
      \bigg(\tfrac{p}{q_3}\smallint_{\frac{p-q_3}{q_3}}^{\infty}\tfrac{s^{\frac{q_3}{p}}}{(s+1)^2}\,ds+1\bigg)^{\!\frac{p}{2q_3}}
     \\&
     \cdot
     \bigg\|
     \exp\Big(\smallint_0^{ {T}} \alpha_u\,du \Big)
     \bigg\|_{L^{q_2}(\P;\R)}
     \left\|
     \left(
     \|X_0\|_{\R^d}^2
     +\int_0^{T} \Big|
     \tfrac{ \beta_s }
     {\exp\left(\smallint_0^s \alpha_u\,du \right)}
     \Big|^2
    \,ds
     \right)^{\!\frac{1}{2}}
     \right\|_{L^{q_3}(\P;\R)}.
   \end{align}
   \end{enumerate}
\end{theorem}
\noindent
Theorem~\ref{thm:moments:hilbert.m.intro}
is an immediate consequence of
Corollary~\ref{c:moments:hilbert.m} below and the proof of
Theorem~\ref{thm:moments:hilbert.m.intro} is therefore omitted.
Corollary~\ref{c:moments:hilbert.m} follows from
Theorem~\ref{thm:moments:hilbert.m} below
which is
the main result of this article and which proves $L^p$-estimates for $p\in[1,\infty)$
of more general functions of It\^o processes in a separable Hilbert space.
Since 
Theorem~\ref{thm:moments:hilbert.m} below essentially generalizes both the differential form of the Gronwall inequality
and Lyapunov's second method for stability of equilibria of ordinary differential equations,
we refer to
Theorem~\ref{thm:moments:hilbert.m} below
as a \textit{stochastic Gronwall-Lyapunov inequality}.

In the literature, there are numerous results which assume that
$\alpha$ in
condition~\eqref{eq:lin.growth.intro.thm} is deterministic.
In that case one can first take $L^p$-norm, $p\in[1,\infty)$ and then
apply the classical Gronwall inequality. This approach imposes strong
assumptions (e.g.,  global monotonicity which is
condition~\eqref{eq:ass.41} in the special case $p=2$, $V_0\equiv 0$, $\bar{V}\equiv 0$)
on the problem under consideration which are not
satisfied by numerous interesting
stochastic differential equations.
To illustrate the power of
the stochastic Gronwall-Lyapunov inequality in
Theorem~\ref{thm:moments:hilbert.m} below,
we
discuss in Section~\ref{sec:applications}
the impact of
Theorem~\ref{thm:moments:hilbert.m}
on the following problems:
\begin{enumerate}[(i)]
  \item (Exponential) moment estimates for stochastic differential equations (SDEs); see Subsection~\ref{sec:moments}
  and Subsection~\ref{sec:exp.moments}.
  \item Strong local Lipschitz continuity in the initial value;
  see Subsection~\ref{sec:localLip}.
  \item Strong completeness of SDEs;
  see Subsection~\ref{sec:completeness}.
  \item Perturbation estimates for SDEs;
  see Subsection~\ref{sec:perturbation}.
\end{enumerate}
In particular, we considerably improve existing results in the literature on these applications;
see Section~\ref{sec:applications} below for details.
Moreover, in the subsequent article Hudde et al.~\cite{HuddeHutzenthalerMazzonetto2019b} we apply
Theorem~\ref{thm:moments:hilbert.m}
and
Corollary~\ref{c:moments:hilbert.m}
to derive versions of solutions of SDEs
which are twice continuously differentiable in the initial value without
assuming the coefficients of the SDE to satisfy a global monotonicity condition.

\subsection{Notation}
\label{sec:notation}

Throughout this article we frequently use the following notation.
For every topological space $(E,\mathcal E)$ 
we denote by $\mathcal{B}(E)$ the Borel-sigma-algebra on  $(E,\mathcal E)$.
For all measurable spaces $(A,\mathcal{A})$ and $(B,\mathcal{B})$
we denote by $\mathcal{M}(\mathcal A,\mathcal B)$ the set of $\mathcal{A}$/$\mathcal{B}$-measurable functions from $A$ to $B$.
For every probability space $(\Omega,\mathcal{A},\P)$, real number $p\in (0,\infty]$, and normed vector space $(V,\|\cdot\|_V)$
we denote by
$\left\|\cdot\right\|_{L^p(\P;V)}\colon\mathcal{M}(\mathcal{A},\mathcal{B}(V))\to[0,\infty]$ the function that
 satisfies for all
$X\in\mathcal{M}(\mathcal{A},\mathcal{B}(V))$ that
$\|X\|_{L^p(\P;V)}=\left(\E\!\left[\|X\|_V^p\right]\right)^{\!\nicefrac1p}$ 
if $p<\infty$ and 
$\|X\|_{L^\infty(\P;V)}
= \inf\{ c \in [0,\infty) \colon \|X\|_V \leq c \ \P\text{-a.s.}\}$
otherwise.
For every $a\in(0,\infty)$ we denote by $\tfrac{a}{0}$
and $\infty^{a}$ the extended real numbers
given by
$\tfrac{a}{0}=\infty$
and $\infty^{a}=\infty$.
 We denote by
 $\tfrac{0}{0}$, $0\cdot\infty$, and  $0^0$
 the extended real numbers given by
 $\tfrac{0}{0}=0$, $0\cdot\infty=0$, and  $0^0=1$.
%
%
%
%
%
%
%
%
%
%
%
For two separable $\R$-Hilbert spaces $( H, \left< \cdot , \cdot \right>_H, \left\| \cdot \right\|_H )$ and $( U, \left< \cdot , \cdot \right>_U, \left\| \cdot \right\|_U )$
and an orthonormal basis $\mathbb{U}$ of $( U, \left< \cdot , \cdot \right>_U, \left\| \cdot \right\|_U )$
we denote $L(U,H)$ the set of continuous linear functions,
by $\|\cdot\|_{\HS(U,H)}\colon L(U,H)\to[0,\infty]$ the function satisfying for all $A\in L(U,H)$
that $\|A\|_{\HS(U,H)}^2=\sum_{u\in\mathbb{U}}\|Au\|_H^2$,
and by $\HS(U,H)$ the set $\HS(U,H)=\{A\in L(U,H)\colon\|A\|_{\HS(U,H)}<\infty\}$.
%
Stochastic integrals with respect to Wiener processes over product measurable and adapted
integrands are defined, e.g., in Weiz\"acker \& Winkler~\cite[Definition 6.3.4]{WeizaeckerWinkler1990}.


\section{A stochastic Gronwall-Lyapunov inequality}
In this section we derive the main result of this paper:
the stochastic Gronwall-Lyapunov inequality
in Theorem~\ref{thm:moments:hilbert.m} below.
First we prove in Lemma~\ref{l:integrating.factor.V}
an almost sure identity for functions of It\^o processes with an
exponential integrating factor.
Moreover, in Lemma~\ref{l:nonlinear.Gronwall}
 we provide an analog of Gronwall's inequality
where the exponential function is replaced by a monomial.
Throughout this section we use the notation from Subsection~\ref{sec:notation}.
\begin{sett} \label{sett:moments.m}
Let
$( H, \left< \cdot , \cdot \right>_H, \left\| \cdot \right\|_H )$
and
$( U, \left< \cdot , \cdot \right>_U, \left\| \cdot \right\|_U )$
be
separable
$\R$-Hilbert spaces,
let 
$T \in (0, \infty)$,
let $(\Omega, \F, \P)$ be a probability space with a normal filtration $(\mathbb{F}_{t})_{t\in [0,T]}$,
let $(W_t)_{t \in [0, T]}$ be an
$\Id_U$-cylindrical $(\mathbb{F}_t)_{t\in[0,T]}$-Wiener process,
let $O\subseteq H$ be an open set,
let $\tau\colon\Omega\to[0,T]$ be a stopping time,
let $X\colon[0,T]\times\Omega\to O$,
$a\colon[0,T]\times\Omega\to H$,
$b\colon[0,T]\times\Omega\to \HS(U,H)$
be $\mathcal{B}([0,T])\otimes\mathcal{F}$-measurable and adapted stochastic processes which satisfy that it holds
$\P$-a.s.\ that $\int_0^{\tau}\|a_s\|_H+\|b_s\|_{\HS(U,H)}^2\,ds<\infty$,
which satisfy that $X$ has continuous sample paths,
and which satisfy that for all $t\in[0,T]$ it holds $\P$-a.s.\ that $X_{\min\{t,\tau\}}=X_0+\int_0^t\1_{[0,\tau]}(s) a_s\,ds +\int_0^t\1_{[0,\tau]}(s)b_s\,dW_s$.
\end{sett}

\subsection{Almost sure identity with an exponential integrating factor}\label{ssec:exp.factor}

The following lemma, Lemma~\ref{l:integrating.factor.V}, slightly
generalizes
Lemma 2.1 in Hutzenthaler \& Jentzen~\cite{HutzenthalerJentzen2014}
to time-dependent test functions.
\begin{lemma}[Exponential integrating factor]\label{l:integrating.factor.V}
   Assume Setting~\ref{sett:moments.m},
   let $ V=(V(t,x))_{t\in[0,T],x\in O} \in C^{ 1,2 }( [0,T]\times O, \R )$,
   and let
   $\chi\colon[0,T]\times\Omega\to\R\cup\{-\infty,\infty\}$,
   $\eta\colon[0,T]\times\Omega\to\HS(U,\R)$
   be $\mathcal{B}([0,T])\otimes\mathcal{F}$-measurable and adapted stochastic processes which satisfy
   $\P$-a.s.\ that $\int_0^{\tau}|\chi_s|+\|\eta_s\|_{\HS(U,\R)}^2\,ds<\infty$.
   Then it holds for all $t\in[0,T]$ that $\P$-a.s.
   \begin{equation}  \begin{split} \label{eq:integrating.factor.V}
     &\tfrac{V({\min\{t,\tau\}},X_{\min\{t,\tau\}})}
     {\exp\left(\smallint_0^{\min\{t,\tau\}} \chi_r-\frac{1}{2}\|\eta_r\|_{\HS(U,\R)}^2\,dr
     +\smallint_0^t\1_{[0,\tau]}(r)\eta_r\,dW_r\right)}
     \\&
     =V(0,X_0)+\int_0^t\1_{[0,\tau]}(s)\tfrac{(\frac{\partial}{\partial x}V)(s,X_s)b_s-V(s,X_s)\eta_s}
     {\exp\left(\smallint_0^s \chi_r-\frac{1}{2}\|\eta_r\|_{\HS(U,\R)}^2\,dr
     +\smallint_0^s\eta_r\,dW_r\right)}
     \,dW_s
     \\&
     +\int_0^{\tau}\tfrac{(\frac{\partial}{\partial s}V)(s,X_s)+
     (\frac{\partial}{\partial x}V)(s,X_s)a_s+\frac{1}{2}\textup{trace}\left(b_s^{*}(\textup{Hess}_xV)(s,X_s)b_s\right)
     +\textup{trace}\left(\eta_s^{*}\left[V(s,X_s)\eta_s-(\frac{\partial}{\partial x}V)(s,X_s)b_s\right]\right)- V(s,X_s)\chi_s}
     {\exp\left(\smallint_0^s \chi_r-\frac{1}{2}\|\eta_r\|_{\HS(U,\R)}^2\,dr
     +\smallint_0^s\eta_r\,dW_r\right)}
     \,ds.
   \end{split}     \end{equation}
\end{lemma}
\begin{proof}
  It\^o's formula
  proves that for all $t\in[0,T]$  it holds $\P$-a.s.\ that
  \begin{equation}  \begin{split}
     &\tfrac{V({\min\{t,\tau\}},X_{\min\{t,\tau\}})}
     {\exp\left(\smallint_0^t\1_{[0,\tau]}(r)\big( \chi_r-\frac{1}{2}\|\eta_r\|_{\HS(U,\R)}^2\big)\,dr
     +\smallint_0^t\1_{[0,\tau]}(r)\eta_r\,dW_r\right)}
     \\&
     =V(0,X_0)+\int_0^t\1_{[0,\tau]}(s)\tfrac{(\frac{\partial}{\partial x}V)(s,X_s)b_s-V(s,X_s)\eta_s}
     {\exp\left(\smallint_0^s \chi_r-\frac{1}{2}\|\eta_r\|_{\HS(U,\R)}^2\,dr
     +\smallint_0^s\eta_r\,dW_r\right)}
     \,dW_s
     \\&
     +\int_0^t\1_{[0,\tau]}(s)\tfrac{(\frac{\partial}{\partial s}V)(s,X_s)+
     (\frac{\partial}{\partial x}V)(s,X_s)a_s
     -V(s,X_s)\left[\chi_s-\frac{1}{2}\|\eta_s\|_{\HS(U,\R)}^2\right]
     +\frac{1}{2}\textup{trace}\left(b_s^{*}(\textup{Hess}_xV)(s,X_s)b_s\right)
     }
     {\exp\left(\smallint_0^s \chi_r-\frac{1}{2}\|\eta_r\|_{\HS(U,\R)}^2\,dr
     +\smallint_0^s\eta_r\,dW_r\right)}
     \,ds
     \\&
     +\int_0^t\1_{[0,\tau]}(s)\tfrac{\frac{1}{2}V(s,X_s)\|\eta_s\|_{\HS(U,\R)}^2
     -
     \textup{trace}\left(\eta_s^{*}
     (\frac{\partial}{\partial x}V)(s,X_s)b_s\right)
     }
     {\exp\left(\smallint_0^s \chi_r-\frac{1}{2}\|\eta_r\|_{\HS(U,\R)}^2\,dr
     +\smallint_0^s\eta_r\,dW_r\right)}
     \,ds.
  \end{split}     \end{equation}
  Combining this with the fact that for all $s\in[0,T]$ it holds that
  \begin{equation}  \begin{split}
    &V(s,X_s)\|\eta_s\|_{\HS(U,\R)}^2
    -\operatorname{trace}\left(\eta_s^{*}(\tfrac{\partial}{\partial x}V)(s,X_s)b_s\right)
    \\&
    =
    \operatorname{trace}\left(\eta_s^{*}V(s,X_s)\eta_s\right)
    -\operatorname{trace}\left(\eta_s^{*}(\tfrac{\partial}{\partial x}V)(s,X_s)b_s\right)
    \\&
    =
    \operatorname{trace}\left(\eta_s^{*}\left[V(s,X_s)\eta_s-
    (\tfrac{\partial}{\partial x}V)(s,X_s)b_s\right]\right)
  \end{split}     \end{equation}
  proves~\eqref{eq:integrating.factor.V}
  and finishes the proof of Lemma~\ref{l:integrating.factor.V}.
\end{proof}
\subsection{A nonlinear Gronwall-Bellman-Opial inequality}\label{ssec:nonlinear.G}
The following Gronwall-Bellman-Opial lemma
(see Opial~\cite{Opial1957})
is well-known under additional continuity assumptions;
see, e.g., Beesack~\cite{Beesack1977} or Hutzenthaler \& Jentzen~\cite[Lemma 2.11]{HutzenthalerJentzen2015}.
\begin{lemma}\label{l:nonlinear.Gronwall}
  Let $t\in[0,\infty)$, $p\in(1,\infty)$, let $x,\beta\colon[0,t]\to[0,\infty]$
  be Borel-measurable functions,
  and assume for all $s\in[0,t]$ that
  \begin{equation}  \begin{split}\label{eq:nonlinear.Gronwall.ass}
    (x_s)^p\leq (x_0)^p+p\int_0^s (x_r)^{p-1}\beta_r\,dr<\infty.
  \end{split}     \end{equation}
  Then it holds that
  \begin{equation}  \begin{split}
    x_t\leq x_0+\int_0^t\beta_r\,dr.
  \end{split}     \end{equation}
\end{lemma}
\begin{proof}[Proof of Lemma~\ref{l:nonlinear.Gronwall}]
  Nonnegativity of $p,x,\beta$,
  the fact that $(\eps/2,\infty)\ni y\mapsto y^{\frac{1}{p}}\in\R$, $\eps\in(0,\infty)$,
  are continuously differentiable functions,
  the chain rule for absolutely continuous functions,
  the fact that $(0,\infty)\ni y\mapsto y^{\frac{1}{p}-1}\in\R$
  is decreasing,
  the fact that $(0,\infty)\ni y\mapsto y^{\frac{p-1}{p}}\in\R$ is increasing,
  and assumption~\eqref{eq:nonlinear.Gronwall.ass} imply
  for all $\eps\in(0,\infty)$ that
  \begin{equation}  \begin{split}
    &\left(\eps+(x_0)^p+p\int_0^t(x_r)^{p-1}\beta_r\,dr\right)^{\!\frac{1}{p}}
    \\&
    =
    \left(\eps+(x_0)^p\right)^{\!\frac{1}{p}}
    +\int_0^t \tfrac{1}{p}\Big(\eps+(x_0)^p+p\int_0^s(x_r)^{p-1}\beta_r\,dr
    \Big)^{\frac{1}{p}-1}p(x_s)^{p-1}\beta_s\,ds
    \\&
    \leq
    \left(\eps+(x_0)^p\right)^{\!\frac{1}{p}}
    +\int_0^t \tfrac{1}{p}\Big(\eps+(x_s)^p\Big)^{\frac{1}{p}-1}
    p(x_s)^{p-1}\beta_s\,ds
    \\&
    \leq
    \left(\eps+(x_0)^p\right)^{\!\frac{1}{p}}
    +\int_0^t \tfrac{1}{p}\Big(\eps + (x_s)^p\Big)^{\frac{1}{p}-1}
    p\Big(\eps+(x_s)^p\Big)^{\frac{p-1}{p}}\beta_s\,ds
    =
    \left(\eps+(x_0)^p\right)^{\!\frac{1}{p}}
    +\int_0^t \beta_s\,ds.
  \end{split}     \end{equation}
  This,
  the fact that the function
  $(0,\infty)\ni y\mapsto y^{\frac{1}{p}}\in\R$
  is increasing and continuous,
  and assumption~\eqref{eq:nonlinear.Gronwall.ass} imply
  that
  \begin{equation}  \begin{split}
    x_t&\leq \left((x_0)^p+p\int_0^t(x_r)^{p-1}\beta_r\,dr\right)^{\!\frac{1}{p}}
    =\lim_{(0,\infty)\ni\eps\to0}
    \left(\eps+(x_0)^p+p\int_0^t(x_r)^{p-1}\beta_r\,dr\right)^{\!\frac{1}{p}}
    \\&\leq
    \lim_{(0,\infty)\ni\eps\to0}
    \left(\left(\eps+(x_0)^p\right)^{\!\frac{1}{p}}
    +\int_0^t \beta_s\,ds\right)
    =
    x_0
    +\int_0^t \beta_s\,ds.
  \end{split}     \end{equation}
  This completes the proof of Lemma~\ref{l:nonlinear.Gronwall}.
\end{proof}

\subsection{A new stochastic Gronwall-Lyapunov inequality}\label{ssec:GL.inequality}

The following theorem, Theorem~\ref{thm:moments:hilbert.m}, 
is the main result of this article and states
our stochastic Gronwall-Lyapunov inequality for It\^o processes.
A central step in the proof of the uniform moment estimate~\eqref{eq:thm:moments:hilbert.m.sup} below
is a strong observation
of Scheutzow \cite{Scheutzow2013} namely -- suitably adapted to our situation --
that
inequality~\eqref{eq:at.end}
implies~\eqref{eq:Mplus.Mminus}
and then a maximal $L^p$-inequality for local martingales of Burkholder \cite{Burkholder1975}
can be applied which together with~\eqref{eq:Mplus.Mminus} eliminates the involved local martingale in
inequality~\eqref{eq:at.end}.
Here instead of Burkholder \cite{Burkholder1975} we apply
Theorem 1.4 of Ba{\~n}uelos \& Os\c{e}kowski \cite{BanuelosOsekowski2014}
which provides the optimal constants
   \begin{equation}  \begin{split}\label{eq:optimal.constants}
     \Big((\tfrac{1}{p}-1)^p+\int_{p^{-1}-1}^{\infty}\tfrac{s^{p-1}}{s+1}\,ds\Big)^{\frac{1}{p}}
     =\Big(\tfrac{1}{p}\int_{\frac{1-p}{p}}^{\infty}\tfrac{s^{p}}{(s+1)^2}\,ds\Big)^{\frac{1}{p}},\quad p\in(0,1),
   \end{split}     \end{equation}
in Burkholder's result.
\begin{theorem}[A stochastic Gronwall-Lyapunov inequality]\label{thm:moments:hilbert.m}
Assume Setting~\ref{sett:moments.m},
  let $p\in[1,\infty)$,
  let 
  $ V=(V(s,x))_{s\in[0,T],x\in O} \in C^{ 1,2 }( [0,T]\times O, [0,\infty) )$,
  let
  $\alpha,\beta \colon [0,T] \times \Omega \to [0,\infty]$
  be $\mathcal{B}([0,T])\otimes\mathcal{F}$/$\mathcal{B}([0,\infty])$-measurable and adapted stochastic processes
  which satisfy $\P$-a.s.\ that $\int_0^{\tau}|\alpha_u|\,du<\infty$
  and which satisfy that $\P$-a.s.\ it holds for Lebesgue-almost all $s\in[0,\tau]$ that
  \begin{equation}  \begin{split}\label{eq:ass.powers2}
    &(\tfrac{\partial}{\partial s}V)(s,X_s)
    +
    (\tfrac{\partial}{\partial x}V)(s,X_s)\,
    a_s
    +\tfrac{1}{2}\textup{trace}\Big(
    b_sb_s^{*}\,
    (\textup{Hess}_xV)(s,X_s)
    \Big)
    +\tfrac{p-1}{2}\tfrac{\|
    (\frac{\partial}{\partial x}V)(s,X_s)\,b_s\|_{\HS(U,\R)}^2}
    {V(s,X_s)}
    \\&
    \leq \alpha_s\,V(s,X_s)+\beta_s.
  \end{split}     \end{equation}
   Then
   \begin{enumerate}[(i)]
     \item\label{item:1}  
   it holds for all $q_1,q_2\in(0,\infty]$
   with $\tfrac{1}{q_1}=\tfrac{1}{q_2}+\tfrac{1}{p}$ that
   \begin{equation}  \begin{split}\label{eq:thm:moments:hilbert.m}
     &\|V(\tau,X_{\tau})\|_{L^{q_1}(\P;\R)}
     \\&\leq
     \bigg\|
     \exp\Big(\smallint_0^{\tau} \alpha_u\,du \Big)
     \bigg\|_{L^{q_2}(\P;\R)}
    \cdot
     \left(
     \Big\|V(0,X_0) \Big\|_{L^p(\P;\R)}
     +\int_0^T
     \Big\|
     \tfrac{ \1_{[0,\tau]}(s)\beta_s }
     {\exp\left(\smallint_0^s \alpha_u\,du \right)}
     \Big\|_{L^p(\P;\R)}\,ds
     \right)
   \end{split}     \end{equation}
   and
   \item\label{item:2}
   it holds for all $q_1,q_2,q_3\in(0,\infty]$
   with $q_3< p$ and $\tfrac{1}{q_1}=\tfrac{1}{q_2}+\tfrac{1}{q_3}$ that
   \begin{equation}  \begin{split}\label{eq:thm:moments:hilbert.m.sup}
     &\left\|\sup_{s\in[0,\tau]}V(s,X_s)\right\|_{L^{q_1}(\P;\R)}
     \leq \bigg(\tfrac{p}{q_3}\smallint_{\frac{p-q_3}{q_3}}^{\infty}\tfrac{s^{\frac{q_3}{p}}}{(s+1)^2}\,ds+1\bigg)^{\frac{p}{q_3}}
     \\&
     \cdot
     \left\|
     \exp\left(\smallint_0^{\tau} \alpha_u\,du \right)
     \right\|_{L^{q_2}(\P;\R)}
     \left\|
     V(0,X_0)
     +\int_0^{\tau}
     \tfrac{ \beta_s }
     {\exp\left(\smallint_0^s \alpha_u\,du \right)}\,ds
     \right\|_{L^{q_3}(\P;\R)}.
   \end{split}     \end{equation}
   \end{enumerate}
\end{theorem}
\begin{proof}[Proof of Theorem~\ref{thm:moments:hilbert.m}]
  Throughout this proof
%
%
  let $\tau_n\colon\Omega\to[0,T]$, $n\in\N$, be the functions which satisfy for all $n\in\N$
  that $\tau_n=\inf\big(\{s\in[0,T]\colon
     V(s,X_s)+\int_0^s\|\tfrac{\partial}{\partial x}V(u,X_u)b_u\|_{\HS(U,\R)}^2\,du
     \geq n\}\cup\{\tau\}\big)$
  and for all $x,y\in\R$ we denote by $x\wedge y\in\R$ the real number
  which satisfies that $x\wedge y=\min\{x,y\}$.
  Lemma~\ref{l:integrating.factor.V}
  (applied for every $\eps\in(0,\infty)$
  with
  $V=([0,T]\times O\ni(t,x)\mapsto (\eps+V(t,x))^p\in\R)$,
  $\chi_s(\omega)=p\alpha_s\1_{[0,\tau(\omega)]}(s)$,
  and
  $\eta_s(\omega)=0$
  for all $s\in[0,T]$, $\omega\in\Omega$
  in the notation of
  Lemma~\ref{l:integrating.factor.V})
  yields that
  for all $t\in[0,T]$, $\eps\in(0,\infty)$
  it holds $\P$-a.s.\ that
   \begin{equation}  \begin{split}\label{eq:integrating.factor2}
     &\tfrac{(\eps+V(\tau\wedge t,X_{\tau\wedge t}))^p}
     {\exp\left(\smallint_0^{\tau\wedge t} p\alpha_r\,dr\right)
     }
     =\left(\eps+V(0,X_0)\right)^{p}
     +\int_0^t\tfrac{p\left(\eps+V(s,X_s)\right)^{p-1}(\frac{\partial}{\partial x}V)(s,X_s)b_s\1_{[0,\tau]}(s)}
     {\exp\left(\smallint_0^s p\alpha_r\,dr\right)
     }
     \,dW_s
     \\&
     +\int_0^{\tau\wedge t}\tfrac{p\left(\eps+V(s,X_s)\right)^{p-1}(\frac{\partial}{\partial s}V)(s,X_s)+
     p\left(\eps+V(s,X_s)\right)^{p-1}(\frac{\partial}{\partial x}V)(s,X_s)a_s+p\left(\eps+V(s,X_s)\right)^{p-1}\frac{1}{2}\textup{trace}\left(b_s^{*}(\textup{Hess}_xV)(s,X_s)b_s\right)
     }
     {\exp\left(\smallint_0^s p\alpha_r\,dr\right)
     }
     \,ds
     \\&
     +\int_0^{\tau\wedge t}\tfrac{\frac{1}{2}p(p-1) \left(\eps+V(s,X_s)\right)^{p-2}
     \left\|(\frac{\partial}{\partial x}V)(s,X_s)b_s\right\|_{\HS(U,\R)}^2
     -\left(\eps+V(s,X_s)\right)^{p}p\alpha_s
     }
     {\exp\left(\smallint_0^s p\alpha_r\,dr\right)
     }
     \,ds.
   \end{split}     \end{equation}
   The growth assumption~\eqref{eq:ass.powers2}
   and nonnegativity of $\alpha$
   yield for all $\eps\in(0,\infty)$ that $\P$-a.s.\ it holds for Lebesgue-almost all $s\in[0,\tau]$,
   that
   \begin{equation}  \begin{split}\label{eq:herea}
   &p\left(\eps+V(s,X_s)\right)^{p-1}(\tfrac{\partial}{\partial s}V)(s,X_s)+
     p\left(\eps+V(s,X_s)\right)^{p-1}(\tfrac{\partial}{\partial x}V)(s,X_s)a_s
     \\&\quad
     +p\left(\eps+V(s,X_s)\right)^{p-1}\tfrac{1}{2}\textup{trace}\left(b_s^{*}(\textup{Hess}_xV)(s,X_s)b_s\right)
     \\&\quad
     +\tfrac{1}{2}p(p-1) \left(\eps+V(s,X_s)\right)^{p-2}
     \left\|(\tfrac{\partial}{\partial x}V)(s,X_s)b_s\right\|_{\HS(U,\R)}^2
     -(\eps+V(s,X_s))^pp\alpha_s
     \\&
   =
   p\left(\eps+V(s,X_s)\right)^{p-1}
   \bigg[
     (\tfrac{\partial}{\partial s}V)(s,X_s)
     +
     (\tfrac{\partial}{\partial x}V)(s,X_s)a_s
     +\tfrac{1}{2}\textup{trace}\left(b_sb_s^{*}(\textup{Hess}_xV)(s,X_s)\right)
     \\&\quad
     +\tfrac{p-1}{2}\tfrac{
     \left\|(\frac{\partial}{\partial x}V)(s,X_s)b_s\right\|_{\HS(U,\R)}^2
     }{\eps+V(s,X_s)}
     -\alpha_s \big(\eps+V(s,X_s)\big)
     \bigg]
     \\&
   \leq
   p\left(\eps+V(s,X_s)\right)^{p-1}
   \beta_s.
 \end{split}     \end{equation}
  Then~\eqref{eq:integrating.factor2}
  and~\eqref{eq:herea}
   imply that for all $t\in[0,T]$, $\eps\in(0,\infty)$, $n\in\N$
   it holds $\P$-a.s.\ that
   \begin{equation}  \begin{split}\label{eq:at.end}
     &\tfrac{\left(\eps+V({\tau_n\wedge t},X_{\tau_n\wedge t})\right)^{p}} {\exp\left(\smallint_0^{\tau_n\wedge t} p\alpha_u\,du \right)}
     \\&
     \leq
     \left(\eps+V(0,X_0)\right)^{p}
     +\int_0^t\tfrac{p(\eps+V(s,X_s))^{p-1}(\frac{\partial}{\partial x}V)(s,X_s)b_s\1_{[0,\tau_n]}(s)}{
     \exp\left(\smallint_0^s p\alpha_u\,du \right)}\,dW_s
     +\int_0^t
     \tfrac{ p(\eps+V(s,X_s))^{p-1}\beta_s \1_{[0,\tau_n]}(s)}
     {\exp\left(\smallint_0^s p\alpha_u\,du \right)}\,ds.
   \end{split}     \end{equation}
   Now we prove item~\eqref{item:1}.
   Without loss of generality we assume that
   $\E\big[|V(0,X_0)|^p\big]<\infty$
   and that
   $\int_0^T
     \|
     \tfrac{ \1_{[0,\tau]}(s)\beta_s }
     {\exp\left(\smallint_0^s \alpha_u\,du \right)}
     \|_{L^p(\P;\R)}\,ds
   <\infty$
   (otherwise the assertion is trivial).
   Note for every $n\in\N$ that $\tau_n$ is a stopping time
   and that the stochastic integral
   on the right-hand side of~\eqref{eq:at.end} stopped at $\tau_n$ is
   integrable with vanishing expectation.
   This,
   \eqref{eq:at.end},
   linearity,
   Tonelli's theorem,
   and
   H\"older's inequality
   yield for all $t\in[0,T]$, $\eps\in(0,\infty)$, $n\in\N$
   that
   \begin{equation}  \begin{split}\label{eq:estimate.for.Gronwall}
     &\Big\|\tfrac{\eps+V({\tau_n\wedge t},X_{\tau_n\wedge t})} {\exp\left(\smallint_0^{\tau_n\wedge t} \alpha_u\,du \right)}
     \Big\|_{L^p(\P;\R)}^p=
     \E\Big[
     \tfrac{\left(\eps+V({\tau_n\wedge t},X_{\tau_n\wedge t})\right)^{\!p}} {\exp\left(\smallint_0^{\tau_n\wedge t} p\alpha_u\,du \right)}
     \Big]
     \\&
     \leq
     \E\!\left[\left(\eps+V(0,X_0)\right)^{\!p}
     +\int_0^{t}\tfrac{p(\eps+V(s,X_s))^{p-1}(\frac{\partial}{\partial x}V)(s,X_s)b_s\1_{[0,\tau_n]}(s)}{
     \exp\left(\smallint_0^s p\alpha_u\,du \right)}\,dW_s
     +\int_0^{t}
     \tfrac{ p(\eps+V(s,X_s))^{p-1}\beta_s\1_{[0,\tau_n]}(s) }
     {\exp\left(\smallint_0^s p\alpha_u\,du \right)}\,ds
     \right]
     \\&
     =
     \Big\|\eps+V(0,X_0)\Big\|_{L^p(\P;\R)}^p
     +p\int_0^{t}
     \E\!\left[
     \tfrac{ (\eps+V(\tau_n\wedge s,X_{\tau_n\wedge s}))^{p-1} }
     {\exp\left(\smallint_0^{\tau_n\wedge s} (p-1)\alpha_u\,du \right)}
     \tfrac{\1_{[0,\tau_n]}(s)\beta_s}
     {\exp\left(\smallint_0^s \alpha_u\,du \right)}
     \right]
     \,ds
     \\&
     \leq
     \Big\|\eps+V(0,X_0)\Big\|_{L^p(\P;\R)}^p
     +p\int_0^{t}
     \Big\|
     \tfrac{ \left(\eps+V(\tau_n\wedge s,X_{\tau_n\wedge s})\right)^{p-1} }
     {\exp\left(\smallint_0^{\tau_n\wedge s} (p-1)\alpha_u\,du \right)}
     \Big\|_{L^{\frac{p}{p-1}}(\P;\R)}
     \Big\|
     \tfrac{ \1_{[0,\tau_n]}(s)\beta_s }
     {\exp\left(\smallint_0^s \alpha_u\,du \right)}
     \Big\|_{L^p(\P;\R)}\,ds
     \\&
     =
     \Big\|\eps+V(0,X_0)\Big\|_{L^p(\P;\R)}^p
     +p\int_0^{t}
     \Big\|
     \tfrac{ \eps+V(\tau_n\wedge s,X_{\tau_n\wedge s}) }
     {\exp\left(\smallint_0^{\tau_n\wedge s} \alpha_u\,du \right)}
     \Big\|_{L^p(\P;\R)}^{p-1}
     \Big\|
     \tfrac{ \1_{[0,\tau_n]}(s)\beta_s }
     {\exp\left(\smallint_0^s \alpha_u\,du \right)}
     \Big\|_{L^p(\P;\R)}\,ds.
   \end{split}     \end{equation}
   Observe for all $t\in[0,T]$, $\eps\in(0,\infty)$, $n\in\N$ that
   \begin{equation}  \begin{split}
     &\Big\|\eps+V(0,X_0)\Big\|_{L^p(\P;\R)}^p
     +p\int_0^{t}
     \Big\|
     \tfrac{ \eps+V(\tau_n\wedge s,X_{\tau_n\wedge s}) }
     {\exp\left(\smallint_0^{\tau_n\wedge s} \alpha_u\,du \right)}
     \Big\|_{L^p(\P;\R)}^{p-1}
     \Big\|
     \tfrac{ \1_{[0,\tau_n]}(s)\beta_s }
     {\exp\left(\smallint_0^s \alpha_u\,du \right)}
     \Big\|_{L^p(\P;\R)}\,ds
     \\&
     \leq
     \Big\|\eps+V(0,X_0)\Big\|_{L^p(\P;\R)}^p
     +p
     \Big\|
      \eps+V(0,X_{0})+n
     \Big\|_{L^p(\P;\R)}^{p-1}
     \int_0^{T}
     \Big\|
     \tfrac{ \1_{[0,\tau]}(s)\beta_s }
     {\exp\left(\smallint_0^s \alpha_u\,du \right)}
     \Big\|_{L^p(\P;\R)}\,ds
     <\infty.
   \end{split}     \end{equation}
   This, \eqref{eq:estimate.for.Gronwall},
   and the nonlinear Gronwall-Bellman-Opial inequality
   in Lemma~\ref{l:nonlinear.Gronwall}
   (applied in the case $p>1$ for all $n\in\N$, $\eps\in(0,\infty)$
   with $t=T$, $p=p$, $x_s=
     \Big\|
     \tfrac{ \eps+V(\tau_n\wedge s,X_{\tau_n\wedge s}) }
     {\exp\left(\smallint_0^{\tau_n\wedge s} \alpha_u\,du \right)}
     \Big\|_{L^p(\P;\R)}
     $,
     $\beta_s=
     \Big\|
     \tfrac{ \1_{[0,\tau_n]}(s)\beta_s }
     {\exp\left(\smallint_0^s \alpha_u\,du \right)}
     \Big\|_{L^p(\P;\R)}$
   for all $s\in[0,T]$ in the notation of
   Lemma~\ref{l:nonlinear.Gronwall})
   imply for all
   $n\in\N$,
   $\eps\in(0,\infty)$
   that
   \begin{equation}  \begin{split}\label{eq:at.end2}
     &\Big\|\tfrac{\eps+V({\tau_n\wedge T},X_{\tau_n\wedge T})} {\exp\left(\smallint_0^{\tau_n\wedge T} \alpha_u\,du \right)}
     \Big\|_{L^p(\P;\R)}
     \leq
     \Big\|\eps+V(0,X_0) \Big\|_{L^p(\P;\R)}
     +\int_0^T
     \Big\|
     \tfrac{ \1_{[0,\tau_n]}(s)\beta_s }
     {\exp\left(\smallint_0^s \alpha_u\,du \right)}
     \Big\|_{L^p(\P;\R)}\,ds.
   \end{split}     \end{equation}
   Moreover the fact that $(V(s,X_s))_{s\in[0,T]}$ and
   $((\tfrac{\partial}{\partial x}V)(s,X_s))_{s\in[0,T]}$ have
   continuous sample paths and the fact that
   $\P(\int_0^{\tau}\|b_s\|_{\HS(U,H)}^2\,ds<\infty)=1$
   imply  that $\P(\tau=\lim_{n\to\infty}\tau_n)=1$.
   This, continuity of $V$, nonnegativity of $V$ and $\beta$,
   Fatou's lemma,
   \eqref{eq:at.end2},
   and the dominated convergence theorem together with
  $\E[|1+V(0,X_0)|^p]<\infty$
   ensure that
   \begin{equation}  \begin{split}
     &\E\Big[
     \tfrac{\left(V({\tau},X_{\tau})\right)^{\!p}} {\exp\left(\smallint_0^{\tau} p\alpha_u\,du \right)}
     \Big]
     =
     \E\Big[
     \liminf_{n\to\infty}
     \tfrac{\left(V({\tau_n},X_{\tau_n})\right)^{\!p}} {\exp\left(\smallint_0^{\tau_n} p\alpha_u\,du \right)}
     \Big]
     \leq
     \liminf_{(0,1]\ni\eps\to 0}\liminf_{n\to\infty}
     \E\Big[
     \tfrac{\left(\eps+V({\tau_n},X_{\tau_n})\right)^{\!p}} {\exp\left(\smallint_0^{\tau_n} p\alpha_u\,du \right)}
     \Big]
     \\&
     \leq
     \liminf_{(0,1]\ni\eps\to 0}\liminf_{n\to\infty}
     \left(
     \Big\|\eps+V(0,X_0) \Big\|_{L^p(\P;\R)}
     +\int_0^T
     \Big\|
     \tfrac{ \1_{[0,\tau_n]}(s)\beta_s }
     {\exp\left(\smallint_0^s \alpha_u\,du \right)}
     \Big\|_{L^p(\P;\R)}\,ds
     \right)^{\!p}
     \\&
     \leq
     \liminf_{(0,1]\ni\eps\to 0}
     \left(
     \Big\|\eps+V(0,X_0) \Big\|_{L^p(\P;\R)}
     +\int_0^T
     \Big\|
     \tfrac{ \1_{[0,\tau]}(s)\beta_s }
     {\exp\left(\smallint_0^s \alpha_u\,du \right)}
     \Big\|_{L^p(\P;\R)}\,ds
     \right)^{\!p}
     \\&
     =
     \left(
     \Big\|V(0,X_0) \Big\|_{L^p(\P;\R)}
     +\int_0^T
     \Big\|
     \tfrac{ \1_{[0,\tau]}(s)\beta_s }
     {\exp\left(\smallint_0^s \alpha_u\,du \right)}
     \Big\|_{L^p(\P;\R)}\,ds
     \right)^{\!p}.
   \end{split}     \end{equation}
   This, the assumption $\P(\int_0^{\tau} \alpha_u\,du<\infty)=1$, and H\"older's inequality
   yield that
   for all $q_1,q_2\in(0,\infty]$
   with $\tfrac{1}{q_1}=\tfrac{1}{q_2}+\tfrac{1}{p}$ 
   it holds that
   \begin{equation}  \begin{split}
     &\|V(\tau,X_\tau)\|_{L^{q_1}(\P;\R)}
     =\Big\|
     \exp\Big(\smallint_0^{\tau} \alpha_u\,du \Big)
     \tfrac{V(\tau,X_\tau)} {\exp\left(\smallint_0^\tau \alpha_u\,du \right)}
     \Big\|_{L^{q_1}(\P;\R)}
     \\&
     \leq
     \Big\|
     \exp\Big(\smallint_0^\tau \alpha_u\,du \Big)
     \Big\|_{L^{q_2}(\P;\R)}
     \cdot
     \Big\|
     \tfrac{V(\tau,X_\tau)} {\exp\left(\smallint_0^\tau \alpha_u\,du \right)}
     \Big\|_{L^{p}(\P;\R)}
     \\&
     \leq
     \bigg\|
     \exp\Big(\smallint_0^{\tau} \alpha_u\,du \Big)
     \bigg\|_{L^{q_2}(\P;\R)}
     \cdot
     \left(
     \Big\|V(0,X_0) \Big\|_{L^p(\P;\R)}
     +\int_0^T
     \Big\|
     \tfrac{ \1_{[0,\tau]}(s)\beta_s }
     {\exp\left(\smallint_0^s \alpha_u\,du \right)}
     \Big\|_{L^p(\P;\R)}\,ds
     \right).
   \end{split}     \end{equation}
   This proves item~\eqref{item:1}.

   Next we prove item~\eqref{item:2}.
   Throughout the proof of item~\eqref{item:2}
   let $q_1,q_2,q_3\in(0,\infty]$ satisfy that $q_3<p$ and $\tfrac{1}{q_1}=\tfrac{1}{q_2}+\tfrac{1}{q_3}$
   and let $M^{\eps}\colon [0,T]\times\Omega\to\R$, $\eps\in(0,\infty)$, be stochastic processes with continuous sample paths
   such that for all $t\in[0,T]$, $\eps\in(0,\infty)$ it holds $\P$-a.s.\ that
   \begin{equation}  \begin{split}
   M_t^{\eps}=
   \int_0^t\tfrac{p(\eps+V(s,X_s))^{p-1}(\frac{\partial}{\partial x}V)(s,X_s)b_s\1_{[0,\tau]}(s)}{ \exp\left(\smallint_0^s p\alpha_u\,du \right)}\,dW_s.
   \end{split}     \end{equation}
   Such a continuous process exists since It\^o-integrals admit continuous versions.
   Throughout the proof of item~\eqref{item:2}
   we assume
   without loss of generality that
   $
     \left\|
     V(0,X_0)
     +\int_0^{\tau}
     \tfrac{ \beta_s }
     {\exp\left(\smallint_0^s \alpha_u\,du \right)}\,ds
     \right\|_{L^{q_3}(\P;\R)}<\infty
   $
   (otherwise the assertion is trivial).
   Now~\eqref{eq:at.end} and nonnegativity of $V$ and $\beta$ imply that
   for all $\eps\in(0,\infty)$, $n\in\N$ it holds $\P$-a.s.\ for all $t\in[0,T]$ that
   \begin{equation}  \begin{split}\label{eq:Mplus.Mminus}
     \max\!\left\{-M_{\tau_n\wedge t}^{\eps},0\right\}
     \leq
     \left(\eps+V(0,X_0)\right)^{p}
     +\int_0^t
     \tfrac{ p(\eps+V(s,X_s))^{p-1}\beta_s\1_{[0,\tau_n]}(s) }
     {\exp\left(\smallint_0^s p\alpha_u\,du \right)}\,ds.
   \end{split}     \end{equation}
   Then~\eqref{eq:at.end}, the triangle inequality,
   Theorem 1.4 in Ba{\~n}uelos \& Os\c{e}kowski \cite{BanuelosOsekowski2014}
   (applied for all $\eps\in(0,\infty)$, $n\in\N$ with
   continuous martingale $(M_{\tau_n\wedge t})_{t\in[0,\infty)}$)
   together with~\eqref{eq:optimal.constants},
   and~\eqref{eq:Mplus.Mminus}
   yield
   for all $n\in\N$, $\eps\in(0,\infty)$, $q\in(0,1)$
   that
   \begin{equation}  \begin{split}\label{eq:prove.item.ii}
     &\E\!\left[\sup_{t\in[0,\tau_n]}\left|\tfrac{\left(\eps+V(t,X_t)\right)^{p}} {\exp\left(\smallint_0^t p\alpha_u\,du \right)}\right|^{q}\right]
     =\E\!\left[\left|\sup_{t\in[0,T]}\tfrac{\left(\eps+V({\tau_n\wedge t},X_{\tau_n\wedge t})\right)^{p}} {\exp\left(\smallint_0^{\tau_n\wedge t} p\alpha_u\,du \right)}\right|^{q}\right]
     \\&
     \leq
     \E\!\left[\left|
     \sup_{t\in[0,\infty)}M_{\tau_n\wedge t}
     +
     \left(\eps\!+V(0,X_0)\right)^{p}
     +\!\int_0^{\tau_n}\!
     \tfrac{ p(\eps+V(s,X_s))^{p-1}\beta_s }
     {\exp\left(\smallint_0^s p\alpha_u\,du \right)}\,ds
     \right|^{q}
     \right]
     \\&
     \leq
     \E\!\left[\left(\sup_{t\in[0,\infty)}
     \max\!\left\{M_{\tau_n\wedge t},0\right\}
     \right)^{\!q}
     \right]
     +
     \E\!\left[
     \left|
     \left(\eps+V(0,X_0)\right)^{p}
     +\int_0^{\tau_n}
     \tfrac{ p(\eps+V(s,X_s))^{p-1}\beta_s }
     {\exp\left(\smallint_0^s p\alpha_u\,du \right)}\,ds
     \right|^q
     \right]
     \\&
     \leq
     \bigg(
     \tfrac{1}{q}\smallint_{\frac{1-q}{q}}^{\infty}\tfrac{s^{q}}{(s+1)^2}\,ds
     \bigg)
     \E\!\left[\left(\sup_{t\in[0,\infty)}
     \max\!\left\{-M_{\tau_n\wedge t},0\right\}
     \right)^{\!q}
     \right]
     \\&\qquad\qquad
     +
     \E\!\left[
     \left|
     \left(\eps+V(0,X_0)\right)^{p}
     +\int_0^{\tau_n}
     \tfrac{ p(\eps+V(s,X_s))^{p-1}\beta_s }
     {\exp\left(\smallint_0^s p\alpha_u\,du \right)}\,ds
     \right|^q
     \right]
     \\&
     \leq
     \bigg(\tfrac{1}{q}\smallint_{\frac{1-q}{q}}^{\infty}\tfrac{s^{q}}{(s+1)^2}\,ds+1\bigg)
     \E\!\left[
     \left|
     \left(\eps+V(0,X_0)\right)^{p}
     +\int_0^{\tau_n}
     \tfrac{ p(\eps+V(s,X_s))^{p-1}\beta_s }
     {\exp\left(\smallint_0^s p\alpha_u\,du \right)}\,ds
     \right|^q
     \right].
   \end{split}     \end{equation}
   The fact that $\forall n\in\N\colon \tau_n\leq \tau$,
   nonnegativity of $\beta$,
   and
   H\"older's inequality
   prove that
   for all
   $n\in\N$, $\eps\in(0,\infty)$, $q\in(0,1)$
   it holds that
   \begin{equation}  \begin{split}\label{eq:sup.trick}
    & \E\!\left[
     \left|
     \left(\eps+V(0,X_0)\right)^{p}
     +\int_0^{\tau_n}
     \tfrac{ p(\eps+V(s,X_s))^{p-1}\beta_s }
     {\exp\left(\smallint_0^s p\alpha_u\,du \right)}\,ds
     \right|^q
     \right]
    \\&
    \leq\E\!\left[
    \left(\sup_{t\in[0,\tau_n]}\tfrac{\left(\eps+V(t,X_t)\right)^{\!(p-1)q}}{\exp\left(\smallint_0^t (p-1)q\alpha_u\,du \right)}\right)
     \left|
     \eps+V(0,X_0)
     +\int_0^{\tau_n}
     \tfrac{ \beta_s }
     {\exp\left(\smallint_0^s \alpha_u\,du \right)}\,ds
     \right|^q
     \right]
    \\&
    \leq
    \Bigg(\E\bigg[
    \sup_{t\in[0,\tau_n]}\tfrac{\left(\eps+V(t,X_t)\right)^{pq}}{\exp\left(\smallint_0^t pq\alpha_u\,du \right)}
    \bigg]\Bigg)^{\!\frac{p-1}{p}}
    \Bigg(
    \E\Bigg[
     \left|
     \eps+V(0,X_0)
     +\int_0^{\tau}
     \tfrac{ \beta_s }
     {\exp\left(\smallint_0^s \alpha_u\,du \right)}\,ds
     \right|^{pq}
     \Bigg]
    \Bigg)^{\!\frac{1}{p}}.
   \end{split}     \end{equation}
   Combining~\eqref{eq:sup.trick} and~\eqref{eq:prove.item.ii}
   proves for all
   $n\in\N$, $\eps\in(0,\infty)$, $q\in(0,1)$
   that
   \begin{equation}  \begin{split}\label{eq:before.dividing}
     &\E\!\left[\sup_{t\in[0,\tau_n]}\left|\tfrac{\left(\eps+V(t,X_t)\right)^{p}} {\exp\left(\smallint_0^t p\alpha_u\,du \right)}\right|^{q}\right]
     \bigg(\tfrac{1}{q}\smallint_{\frac{1-q}{q}}^{\infty}\tfrac{s^{q}}{(s+1)^2}\,ds+1\bigg)^{-1}
    \\&
    \leq
    \Bigg(\E\bigg[
    \sup_{t\in[0,\tau_n]}\tfrac{\left(\eps+V(t,X_t)\right)^{pq}}{\exp\left(\smallint_0^t pq\alpha_u\,du \right)}
    \bigg]\Bigg)^{\!\frac{p-1}{p}}
    \Bigg(
    \E\Bigg[
     \left|
     \eps+V(0,X_0)
     +\int_0^{\tau}
     \tfrac{ \beta_s }
     {\exp\left(\smallint_0^s \alpha_u\,du \right)}\,ds
     \right|^{pq}
     \Bigg]
    \Bigg)^{\!\frac{1}{p}}.
   \end{split}     \end{equation}
   Note that $q_3/p<1$.
   Next dividing for every $n\in\N$, $\eps\in(0,\infty)$, $q\in(0,1)$
   with $pq\leq q_3$
   both sides of~\eqref{eq:before.dividing} by
   \begin{equation}  \begin{split}
   \bigg(\E\bigg[
    \sup_{t\in[0,\tau_n]}\tfrac{\left(\eps+V(t,X_t)\right)^{\!pq}}{\exp\left(\smallint_0^t pq\alpha_u\,du \right)}
    \bigg]\bigg)^{\!\frac{p-1}{p}}
    \in\Big[\eps^{(p-1)q},\Big(\E\Big[(\eps+V(0,X_0)+n)^{pq}\Big]\Big)^{\!\frac{p-1}{p}}\Big]
    \subseteq(0,\infty)
   \end{split}     \end{equation}
   shows for all $q\in(0,q_3/p]$ that
   \begin{equation}  \begin{split}
     &\left(\E\!\left[\sup_{t\in[0,\tau_n]}\left|\tfrac{\eps+V(t,X_t)} {\exp\left(\smallint_0^t \alpha_u\,du \right)}\right|^{pq}\right]
     \right)^{\!\frac{1}{pq}}
     \bigg(\tfrac{1}{q}\smallint_{\frac{1-q}{q}}^{\infty}\tfrac{s^{q}}{(s+1)^2}\,ds+1\bigg)^{-\frac{1}{q}}
    \\&
    \leq
    \Bigg(
    \E\Bigg[
     \left|
     \eps+V(0,X_0)
     +\int_0^{\tau}
     \tfrac{ \beta_s }
     {\exp\left(\smallint_0^s \alpha_u\,du \right)}\,ds
     \right|^{pq}
     \Bigg]
    \Bigg)^{\frac{1}{pq}}.
   \end{split}     \end{equation}
   Then
   the monotone convergence theorem,
   the fact that
   $
     \big\|
     V(0,X_0)
     +\int_0^{\tau}
     \tfrac{ \beta_s }
     {\exp\left(\smallint_0^s \alpha_u\,du \right)}\,ds
     \big\|_{L^{q_3}(\P;\R)}<\infty
   $,
   and the dominated convergence theorem 
   prove that
   \begin{equation}  \begin{split}
     &\left\|\sup_{t\in[0,\tau]}\tfrac{V(t,X_t)} {\exp\left(\smallint_0^t \alpha_u\,du \right)}\right\|_{L^{q_3}(\P;\R)}
     \leq\liminf_{(0,\infty)\ni \eps\to0}
     \left(\E\!\left[\lim_{\N\ni n\to\infty}\sup_{t\in[0,\tau_n]}\left|\tfrac{\eps+V(t,X_t)} {\exp\left(\smallint_0^t \alpha_u\,du \right)}\right|^{q_3}\right]
     \right)^{\!\frac{1}{q_3}}
     \\&
     =\lim_{(0,\infty)\ni \eps\to0}\lim_{\N\ni n\to\infty}
     \left(\E\!\left[\sup_{t\in[0,\tau_n]}\left|\tfrac{\eps+V(t,X_t)} {\exp\left(\smallint_0^t \alpha_u\,du \right)}\right|^{q_3}\right]
     \right)^{\!\frac{1}{q_3}}
    \\&
    \leq
     \bigg(\tfrac{1}{q_3/p}\smallint_{\frac{1-q_3/p}{q_3/p}}^{\infty}\tfrac{s^{q_3/p}}{(s+1)^2}\,ds+1\bigg)^{\frac{1}{q_3/p}}
    \lim_{(0,\infty)\ni \eps\to0}\Bigg(
    \E\Bigg[
     \left|
     \eps+V(0,X_0)
     +\int_0^{\tau}
     \tfrac{ \beta_s }
     {\exp\left(\smallint_0^s \alpha_u\,du \right)}\,ds
     \right|^{q_3}
     \Bigg]
    \Bigg)^{\frac{1}{q_3}}
    \\&
    =
     \bigg(\tfrac{p}{q_3}\smallint_{\frac{p-q_3}{q_3}}^{\infty}\tfrac{s^{\frac{q_3}{p}}}{(s+1)^2}\,ds+1\bigg)^{\frac{p}{q_3}}
     \Bigg(
    \E\Bigg[
     \left|
     V(0,X_0)
     +\int_0^{\tau}
     \tfrac{ \beta_s }
     {\exp\left(\smallint_0^s \alpha_u\,du \right)}\,ds
     \right|^{q_3}
     \Bigg]
    \Bigg)^{\frac{1}{q_3}}.
   \end{split}     \end{equation}
   Finally, this, nonnegativity of $\alpha$, and H\"older's inequality prove that
   \begin{align}\nonumber
     &\left\|\sup_{t\in[0,\tau]}V(t,X_t)\right\|_{L^{q_1}(\P;\R)}
     \leq \left\|
     \exp\left(\smallint_0^{\tau} \alpha_u\,du \right)
     \sup_{t\in[0,\tau]}\tfrac{V(t,X_t)} {\exp\left(\smallint_0^t \alpha_u\,du \right)}
     \right\|_{L^{q_1}(\P;\R)}
     \\&
     \leq \left\|
     \exp\left(\smallint_0^{\tau} \alpha_u\,du \right)
     \right\|_{L^{q_2}(\P;\R)}
     \left\|\sup_{t\in[0,\tau]}\tfrac{V(t,X_t)} {\exp\left(\smallint_0^t \alpha_u\,du \right)}\right\|_{L^{q_3}(\P;\R)}
     \\&
     \leq \left\|
     \exp\left(\smallint_0^{\tau} \alpha_u\,du \right)
     \right\|_{L^{q_2}(\P;\R)}
     \bigg(\tfrac{p}{q_3}\smallint_{\frac{p-q_3}{q_3}}^{\infty}\tfrac{s^{\frac{q_3}{p}}}{(s+1)^2}\,ds+1\bigg)^{\frac{p}{q_3}}
     \left\|
     V(0,X_0)
     +\int_0^{\tau}
     \tfrac{ \beta_s }
     {\exp\left(\smallint_0^s \alpha_u\,du \right)}\,ds
     \right\|_{L^{q_3}(\P;\R)}.\nonumber
   \end{align}
This establishes item~\eqref{item:2}
and completes the proof of Theorem~\ref{thm:moments:hilbert.m}.
\end{proof}

\begin{cor}[A stochastic Gronwall inequality for It\^o processes]\label{c:moments:hilbert.m}
Assume Setting~\ref{sett:moments.m},
let $p\in[2,\infty)$,
and
let
$\alpha,\beta \colon [0,T] \times \Omega \to [0,\infty]$
be $\mathcal{B}([0,T])\otimes\mathcal{F}$/$\mathcal{B}([0,\infty])$-measurable and adapted stochastic processes
which satisfy $\P$-a.s.\ that
$\int_0^{\tau}|\alpha_u|\,du<\infty$
and which satisfy that $\P$-a.s.\ it holds for Lebesgue-almost all $t\in[0,\tau]$ that
\begin{equation}  \begin{split}\label{eq:lin.growth.c}
  \langle X_t,a_t\rangle_H+\tfrac{1}{2}\|b_t\|_{\HS(U,H)}^2
  +\tfrac{p-2}{2}\tfrac{\|\langle X_t,b_t\rangle_H\|_{\HS(U,\R)}^2}{\|X_t\|_H^2}
  \leq \alpha_t\|X_t\|_H^2+\tfrac{1}{2}|\beta_t|^2.
\end{split}     \end{equation}
 Then
   \begin{enumerate}[(i)]
     \item\label{item:lingrowth1}
   it holds for all $q_1,q_2\in(0,\infty]$
   with $\tfrac{1}{q_1}=\tfrac{1}{q_2}+\tfrac{1}{p}$ that
   \begin{equation}  \begin{split}\label{eq:c:moments.hilbert.m}
     &\|X_{\tau}\|_{L^{q_1}(\P;H)}
     \leq
     \bigg\|
     \exp\Big(\smallint_0^{ {\tau}} \alpha_u\,du \Big)
     \bigg\|_{L^{q_2}(\P;\R)}
     \left(
     \left\|
     X_0
     \right\|_{L^p(\P;H)}^2
     +\int_0^T
     \Big\|
     \tfrac{ \1_{[0,\tau]}(s)\beta_s }
     {\exp\left(\smallint_0^s \alpha_u\,du \right)}
     \Big\|_{L^{p}(\P;\R)}^2\,ds
     \right)^{\!\frac{1}{2}}
   \end{split}     \end{equation}
   and
   \item\label{item:lingrowth2}
   it holds for all $q_1,q_2,q_3\in(0,\infty]$
   with $q_3< p$ and $\tfrac{1}{q_1}=\tfrac{1}{q_2}+\tfrac{1}{q_3}$ that
   \begin{align}\label{eq:c:moments.hilbert.m.sup}
     &\left\|\sup_{s\in[0,\tau]}\|X_s\|_{H}\right\|_{L^{q_1}(\P;\R)}
     \leq
      \bigg(\tfrac{p}{q_3}\smallint_{\frac{p-q_3}{q_3}}^{\infty}\tfrac{s^{\frac{q_3}{p}}}{(s+1)^2}\,ds+1\bigg)^{\frac{p}{2q_3}}
     \\&
     \cdot
     \bigg\|
     \exp\Big(\smallint_0^{ {\tau}} \alpha_u\,du \Big)
     \bigg\|_{L^{q_2}(\P;\R)}
     \left\|
     \left(
     \|X_0\|_{H}^2
     +\int_0^{\tau} 
     \Big|
     \tfrac{ \beta_s }
     {\exp\left(\smallint_0^s \alpha_u\,du \right)}
     \Big|^2\,ds
     \right)^{\!\frac{1}{2}}
     \right\|_{L^{q_3}(\P;\R)}.
   \end{align}
   \end{enumerate}
\end{cor}
\begin{proof}[Proof of Corollary~\ref{c:moments:hilbert.m}]
Assumption~\eqref{eq:lin.growth.c} implies that
$\P$-a.s.\ it holds for Lebesgue-almost all $t\in[0,\tau]$ that
\begin{equation}  \begin{split}
 &2 \langle X_t,a_t\rangle_H
     +\tfrac{1}{2}\textup{trace}\left(b_sb_s^{*}2\textup{Id}_H\right)
  +\tfrac{\frac{p}{2}-1}{2}\tfrac{\|2\langle X_t,b_t\rangle_H\|_{\HS(U,\R)}^2}{\|X_t\|_H^2}
  \\&
 =2\left[ \langle X_t,a_t\rangle_H+\tfrac{1}{2}\|b_t\|_{\HS(U,H)}^2
  +\tfrac{p-2}{2}\tfrac{\|\langle X_t,b_t\rangle_H\|_{\HS(U,\R)}^2}{\|X_t\|_H^2}
  \right]
  \leq 2\alpha_t\|X_t\|_H^2+|\beta_t|^2.
\end{split}     \end{equation}
Theorem~\ref{thm:moments:hilbert.m}
(applied with
$p=\tfrac{p}{2}$,
$V(s,x)=\|x\|_{H}^2$, $\alpha_s=2\alpha_s$, $\beta_s=|\beta_s|^2$,
$q_1=\tfrac{q_1}{2}$,
$q_2=\tfrac{q_2}{2}$,
$q_3=\tfrac{q_3}{2}$
for all $s\in[0,T]$, $x\in O$
in the notation of
Theorem~\ref{thm:moments:hilbert.m})
yields that
   it holds for all $q_1,q_2\in(0,\infty]$
   with  $\tfrac{1}{q_1}=\tfrac{1}{q_2}+\tfrac{1}{p}$ that
\begin{equation}  \begin{split}
     &\|X_{{\tau}}\|_{L^{q_1}(\P;H)}^2
     =\left\|\|X_{{\tau}}\|_H^2\right\|_{L^{\frac{{q_1}}{2}}(\P;H)}
     \\&
     \leq
     \bigg\|
     \exp\Big(\smallint_0^{\tau}2 \alpha_u\,du \Big)
     \bigg\|_{L^{\frac{q_2}{2}}(\P;\R)}
     \left(
     \left\|
     \|X_0\|_{H}^2
     \right\|_{L^{\frac{p}{2}}(\P;\R)}
     +\int_0^T
     \Big\|
     \tfrac{ \1_{[0,\tau]}(s)|\beta_s|^2 }
     {\exp\left(\smallint_0^s 2\alpha_u\,du \right)}
     \Big\|_{L^{\frac{p}{2}}(\P;\R)}\,ds
     \right)
     \\&
     =
     \bigg\|
     \exp\Big(\smallint_0^{ {\tau}} \alpha_u\,du \Big)
     \bigg\|_{L^{q_2}(\P;\R)}^2
     \left(
     \left\|
     X_0
     \right\|_{L^p(\P;H)}^2
     +\int_0^T
     \Big\|
     \tfrac{ \1_{[0,\tau]}(s)\beta_s }
     {\exp\left(\smallint_0^s \alpha_u\,du \right)}
     \Big\|_{L^{p}(\P;\R)}^2\,ds
     \right)
\end{split}     \end{equation}
and yields that
   it holds for all $q_1,q_2,q_3\in(0,\infty]$
   with $q_3< p$ and $\tfrac{1}{q_1}=\tfrac{1}{q_2}+\tfrac{1}{q_3}$ that
   \begin{equation}  \begin{split}
     &\Big\|\sup_{s\in[0,\tau]}\|X_s\|_{H}\Big\|_{L^{q_1}(\P;\R)}^2
     =\Big\|\sup_{s\in[0,\tau]}\|X_s\|_{H}^2\Big\|_{L^{\frac{q_1}{2}}(\P;\R)}
     \leq \bigg(\tfrac{p/2}{q_3/2}\smallint_{\frac{p/2-q_3/2}{q_3/2}}^{\infty}\tfrac{s^{\frac{q_3/2}{p/2}}}{(s+1)^2}\,ds+1\bigg)^{\frac{p/2}{q_3/2}}
     \\&
     \cdot
     \bigg\|
     \exp\Big(\smallint_0^{ {\tau}}2 \alpha_u\,du \Big)
     \bigg\|_{L^{\frac{q_2}{2}}(\P;\R)}
     \left\|
     \|X_0\|_{H}^2
     +\int_0^{\tau}
     \tfrac{ |\beta_s|^2 }
     {\exp\left(\smallint_0^s 2\alpha_u\,du \right)}
     \,ds
     \right\|_{L^{\frac{q_3}{2}}(\P;\R)}
     \\&
     = \bigg(\tfrac{p}{q_3}\smallint_{\frac{p-q_3}{q_3}}^{\infty}\tfrac{s^{\frac{q_3}{p}}}{(s+1)^2}\,ds+1\bigg)^{\frac{p}{q_3}}
     \\&
     \cdot
     \bigg\|
     \exp\Big(\smallint_0^{ {\tau}} \alpha_u\,du \Big)
     \bigg\|_{L^{q_2}(\P;\R)}^2
     \left\|
     \left(
     \|X_0\|_{H}^2
     +\int_0^{\tau}
     \Big|
     \tfrac{ \beta_s }
     {\exp\left(\smallint_0^s \alpha_u\,du \right)}
     \Big|^2\,ds
     \right)^{\!\frac{1}{2}}
     \right\|_{L^{q_3}(\P;\R)}^2.
   \end{split}     \end{equation}
  This completes the proof of Corollary~\ref{c:moments:hilbert.m}.
\end{proof}

\section{Applications of the stochastic Gronwall-Lyapunov inequality}
\label{sec:applications}

In this section we apply the stochastic Gronwall-Lyapunov inequality,
i.e.\ Theorem~\ref{thm:moments:hilbert.m},
to improve existing results on
moment estimates for SDEs in Subsection~\ref{sec:moments},
exponential moment estimates for SDEs in Subsection~\ref{sec:exp.moments},
strong local Lipschitz continuity in the initial value in
Subsection~\ref{sec:localLip},
strong completeness for SDEs in 
Subsection~\ref{sec:completeness},
and strong perturbation estimates in
Subsection~\ref{sec:perturbation}.
Throughout this section we use the notation from Subsection~\ref{sec:notation}
and we frequently use the following setting.
\begin{sett}\label{sett:applications}
Let 
  $( H, \left< \cdot , \cdot \right>_H, \left\| \cdot \right\|_H )$
  and $( U, \left< \cdot , \cdot \right>_U, \left\| \cdot \right\|_U )$ be separable $\R$-Hilbert spaces,
  let $T\in (0,\infty)$,
let $(\Omega, \F, \P)$ be a probability space with a normal filtration $(\mathbb{F}_{t})_{t\in [0,T]}$,
let $(W_t)_{t \in [0, T]}$ be an
$\Id_U$-cylindrical $(\mathbb{F}_t)_{t\in[0,T]}$-Wiener process,
  let $O \subseteq H$ be an open set,
  let $\mathcal{O}\in\mathcal{B}(O)$,
  let $ \mu \colon [0,T] \times \mathcal{O} \to H $
  and 
  $ \sigma \colon [0,T] \times \mathcal{O} \to \HS(U,H)$
  be Borel measurable functions,
  let $\tau\colon\Omega\to[0,T]$ be a stopping time,
  and
  let
  $
    X \colon [0,T] \times \Omega \to \mathcal{O}
  $
  be an adapted stochastic
  process with continuous sample paths
  which satisfies that $ \P $-a.s.\ it holds that
  $
    \smallint_0^{ \tau } \| \mu( s, X_s ) \|_H
    + \| \sigma( s, X_s ) \|_{\HS(U,H)}^2
    \, ds < \infty
  $
  and which satisfies that for all $t\in[0,T]$ it holds $\P$-a.s.\ that
  \begin{equation}  \begin{split}\label{eq:appl.X}
    X_{\min\{t,\tau\} } = 
    X_0
    + \smallint_0^{ t }\1_{[0,\tau]}(s) \mu(s, X_s ) \, ds
    +
    \smallint_0^{ t }\1_{[0,\tau]}(s) \sigma(s, X_s ) \, dW_s.
  \end{split}     \end{equation}
\end{sett}

\subsection{Moment estimates for SDEs}\label{sec:moments}
The following corollary, Corollary~\ref{c:powers.norm},
provides marginal and uniform Lyapunov-type estimates for solutions
of SDEs.
The
marginal Lyapunov-type estimate~\eqref{eq:moments.marginal} below
is essentially known in the literature;
see, e.g., Cox et al.~\cite[Lemma 2.2]{CoxHutzenthalerJentzen2013}
or Gy\"ongy \& Krylov \cite[Lemma 2.2]{GyoengyKrylov1996}.
To the best of our knowledge, the uniform Lyapunov-type esimate~\eqref{eq:moments.uniform} below is new.
In the literature, uniform moment estimates are derived
with the help of a Burkholder-Davis-Gundy inequality.
In a number of situations, finiteness of uniform $L^p$-moments could not be established for all $p\in(1,\infty)$
for which marginal $L^p$-moments are finite; e.g.\ for the $3/2$-model of Heston~\cite{Heston1997} and Platen~\cite{Platen1997}
or for the $4/2$-model of Grasselli~\cite{Grasselli2017}.
In many situations where upper bounds for uniform moments could be established,
these 
are less sharp than~\eqref{eq:moments.uniform}; see, e.g., Proposition 2.27 in
Cox et al.~\cite{CoxHutzenthalerJentzen2013}
(with $V$ depending only on the first component).
  Corollary~\ref{c:powers.norm}
  follows directly from
  Corollary~\ref{c:moments:hilbert.m}
  (applied with
  $a_s=\mu(s,X_s)$,
  $b_s=\sigma(s,X_s)$,
  $\alpha_s=\alpha$,
  $\beta_s=\beta$,
  $q_2=\infty$
  for all $s\in[0,T]$
  in the notation of
  Corollary~\ref{c:moments:hilbert.m}).
\begin{cor}[Moment estimates for SDEs]\label{c:powers.norm}
  Assume Setting~\ref{sett:applications},
  let $p\in[2,\infty)$, $\alpha,\beta\in[0,\infty)$,
  and assume that
  for all $t\in[0,T]$, $x\in \mathcal{O}$
  it holds that
  \begin{equation}  \begin{split}\label{eq:ass.powers.norm}
    &
    \langle x,
    \mu(t,x)\rangle_H
    +\tfrac{1}{2}\|\sigma(t,x)\|_{\HS(U,H)}^2
    +\tfrac{p-2}{2}\tfrac{\|
    \langle x,\sigma(t,x)\rangle_H\|_{\HS(U,\R)}^2}
    {\|x\|_H^2}
    \leq \alpha\|x\|^2_H+\tfrac{1}{2}\beta^2.
  \end{split}     \end{equation}
Then
\begin{enumerate}[(i)]
  \item it holds that
  \begin{equation}  \begin{split}\label{eq:moments.marginal}
    \left\|X_{\tau}\right\|_{L^p(\P;H)}
    \leq 
    e^{\alpha T}
    \left(
    \left\|X_0\right\|_{L^p(\P;H)}^2+\int_0^T \tfrac{\beta^2}{e^{2\alpha s}}\,ds\right)^{\!\frac{1}{2}}
  \end{split}     \end{equation}
  and
  \item  it holds for all $q\in(0,p)$ that
  \begin{equation}  \begin{split}\label{eq:moments.uniform}
    \left\|\sup_{t\in[0,\tau]}\left\|X_t\right\|_H\right\|_{L^q(\P;\R)}
    \leq
    e^{\alpha T}
    \left\|\left(\|X_0\|_H^2+\int_0^T\tfrac{\beta^2}{e^{2\alpha s}}\,ds\right)^{\!\frac{1}{2}}\right\|_{L^q(\P;\R)}
    \left(\tfrac{p}{q}\int_{\frac{p-q}{q}}^{\infty}\tfrac{s^{\frac{q}{p}}}{(s+1)^2}\,ds+1\right)^{\!\frac{p}{q}}.
  \end{split}     \end{equation}
\end{enumerate}
\end{cor}

\subsection{Exponential moment estimates for SDEs}\label{sec:exp.moments}
For a number of problems involving SDEs with non-globally monotone coefficients,
it is useful to estimate exponential moments;
see, e.g., Subsection~\ref{sec:localLip}
or Subsection~\ref{sec:perturbation} below.
The following corollary, Corollary~\ref{c:exp.moments},
applies Theorem~\ref{thm:moments:hilbert.m}
to derive suitable exponential moment estimates.
Condition~\eqref{eq:condition.exp.moment}
and the marginal exponential moment estimate~\eqref{eq:exp:mom:estimate}
have been found and derived in
Cox et al.~\cite[Proposition 2.3 and Corollary 2.4]{CoxHutzenthalerJentzen2013}.
To the best of our knowledge, the uniform exponential
moment estimate~\eqref{eq:exp:mom:uniform} is new.

\begin{cor}[Exponential moment estimates for SDEs]\label{c:exp.moments}
  Assume Setting~\ref{sett:applications},
  let $\bar{U}\colon[0,T]\times \mathcal{O}\to\R$ be Borel-measurable
  and satisfy $\int_0^{\tau}|\bar{U}(s,X_s)|\,ds<\infty$,
  let 
  $ U=(U(s,x))_{s\in[0,T],x\in O} \in C^{ 1,2 }( [0,T]\times O, [0,\infty) ) $,
  let $\alpha\in\R$,
  and assume that
  for all $(s,x)\in\cup_{\omega\in\Omega}\{(t,X_t(\omega))\in[0,T]\times \mathcal{O}\colon t\in[0,\tau(\omega)]\}$
  it holds that
  \begin{equation}  \begin{split}\label{eq:condition.exp.moment}
   (\tfrac{\partial}{\partial s}U)(s,x)
    +
    (\tfrac{\partial}{\partial x}U)(s,x)\,
    \mu(s,x)
    &+\tfrac{1}{2}\textup{trace}\Big(
    (\sigma\cdot\sigma^*)(s,x)\,
    (\textup{Hess}_xU)(s,x)
    \Big)
    \\&
    +
    \tfrac{1}{2e^{\alpha s}}\big\|(\tfrac{\partial}{\partial x}U)(s,x)\sigma(s,x)\big\|_{\HS(U,\R)}^2
    +\bar{U}(s,x)\leq \alpha U(s,x).
  \end{split}     \end{equation}
  Then
\begin{enumerate}[(i)]
  \item\label{item:expmom1}
  it holds that
  \begin{equation}
  \label{eq:exp:mom:estimate}
  \begin{split}
  &
    \E\!\left[
      \exp\!\left(
          \tfrac{U( \tau, X_{ \tau } )}{\exp(\alpha \tau)}
        +
        \smallint_0^{ \tau }
            \tfrac{\bar{U}( s, X_s ) }{\exp(\alpha s)}
        \, ds
      \right)
    \right]
  \leq
    \E\Big[\!
      \exp\!\big( 
        U(0,X_0) 
      \big)
    \Big],
  \end{split}
  \end{equation}
  and
  \item\label{item:expmom2}
  it holds for all $q\in(0,1)$ that
  \begin{equation}
  \label{eq:exp:mom:uniform}
  \begin{split}
  &
    \E\!\left[
    \sup_{t\in[0,\tau]}
      \exp\!\left(
          \tfrac{qU( t, X_{ t } )}{\exp(\alpha t)}
        +
        \smallint_0^{ t }
            \tfrac{q\bar{U}( s, X_s ) }{\exp(\alpha s)}
        \, ds
      \right)
    \right]
  \leq
    \E\Big[\!
      \exp\!\big( 
        qU(0,X_0) 
      \big)
    \Big]
    \Big(
    \tfrac{1}{q}\smallint_{\frac{1-q}{q}}^{\infty}\tfrac{s^{q}}{(s+1)^2}\,ds+1\Big).
  \end{split}
  \end{equation}
\end{enumerate}
\end{cor}
\begin{proof}[Proof of Corollary~\ref{c:exp.moments}]
  Throughout this proof 
  let $V\colon[0,T]\times O\times\R\to[0,\infty)$
  and $Y\colon[0,T]\times\Omega\to\R$
  satisfy
  for all $t\in[0,T]$, $x\in O$, $y\in\R$
  that $V(t,x,y)=\exp(U(t,x)e^{-\alpha t}+y)$
  and $Y_t=\int_0^{\min\{t,\tau\}}\bar{U}(s,X_s)e^{-\alpha s}\,ds$.
  Assumption~\eqref{eq:condition.exp.moment}
  yields for all $s\in[0,\tau]$ that
  \begin{equation}  \begin{split}
    &(\tfrac{\partial}{\partial s}V)(s,X_s,Y_s)
    +
    (\tfrac{\partial}{\partial (x,y)}V)(s,X_s,Y_s)\,
    (\mu(s,X_s), \bar{U}(s,X_s)e^{-\alpha s})
    \\&\quad
    +\tfrac{1}{2}\textup{trace}\Big(
    (\sigma(s,X_s), 0)(\sigma(s,X_s), 0)^{*}\,
    (\textup{Hess}_{(x,y)}V)(s,X_s,Y_s)
    \Big)
   \\&
   =V(s,X_s,Y_s)e^{-\alpha s}
   (\tfrac{\partial}{\partial s}U)(s,X_s)
   -V(s,X_s,Y_s)e^{-\alpha s}\alpha U(s,X_s)
    \\&\quad
    +
   V(s,X_s,Y_s)e^{-\alpha s}(\tfrac{\partial}{\partial x}U)(s,X_s)\mu(s,X_s)
   +V(s,X_s,Y_s)e^{-\alpha s}\bar{U}(s,X_s)
    \\&\quad
    +V(s,X_s,Y_s)e^{-\alpha s}\tfrac{1}{2}\textup{trace}\Big(
    \sigma(s,X_s)(\sigma(s,X_s))^{*}\,
    (\textup{Hess}_{x}U)(s,X_s)
    \Big)
    \\&\quad
    +V(s,X_s,Y_s)e^{-2\alpha s}\tfrac{1}{2}\left\|(\tfrac{\partial}{\partial x}U)(s,X_s)\sigma(s,x)\right\|_{\HS(U,\R)}^2
   \\&
   =V(s,X_s,Y_s)e^{-\alpha s}
   \Big(
   (\tfrac{\partial}{\partial s}U)(s,X_s)
    +
    (\tfrac{\partial}{\partial x}U)(s,X_s)\,
    \mu(s,X_s)
    - \alpha U(s,X_s)
    +\bar{U}(s,X_s)
    \\&\qquad\qquad
    +
    \tfrac{1}{2e^{\alpha s}}\big\|(\tfrac{\partial}{\partial x}U)(s,X_s)\sigma(s,X_s)\big\|_{\HS(U,\R)}^2
    +\tfrac{1}{2}\textup{trace}\Big(
    (\sigma\cdot\sigma^*)(s,X_s)\,
    (\textup{Hess}_xU)(s,X_s)
    \Big)
   \Big)
   \\&
   \leq 0.
  \end{split}     \end{equation}
  This and Theorem~\ref{thm:moments:hilbert.m}
  (applied with $H=H\times\R$,
  $\langle (x,r),(y,s)\rangle_H=\langle x,y\rangle_H+ rs$,
  $X_t=(X_t,Y_t)$, $a_t=\1_{[0,\tau]}(t)(\mu(t,X_t),\bar{U}(t,X_t)e^{-\alpha t})$, $b_t=\1_{[0,\tau]}(t)(\sigma(t,X_t),0)$, $V=V$,
  $\alpha_t=0$, $\beta_t=0$, $p=1$, $q_2=\infty$
  for all $x,y\in H$, $r,s\in R$, $t\in[0,T]$
  in the notation of Theorem~\ref{thm:moments:hilbert.m})
  yield item~\eqref{item:expmom1} and item~\eqref{item:expmom2}.
  This finishes the proof of Corollary~\ref{c:exp.moments}.
\end{proof}


Clearly if the solution of
an SDE has finite exponential moments,
then it has finite moments if the starting point
has sufficient  exponential moments.
The marginal moment estimate~\eqref{eq:moment.exp.marginal.condition}
and
the uniform moment estimate~\eqref{eq:moment.exp.uniform.condition}
below
show that it suffices that the starting point has suitable
finite moments if a suitable special case of the exponential moment condition~\eqref{eq:condition.exp.moment}
is satisfied.
In particular, we show in the proof of 
Corollary~\ref{c:exp.moments.condition} below that the exponential moment condition~\eqref{eq:condition.exp.moment.condition}
below implies the moment condition~\eqref{eq:ass.powers.condition} below.

\begin{cor}[Exponential moment condition implies moments]\label{c:exp.moments.condition}
  Assume Setting~\ref{sett:applications},
  let 
  $ U=(U(s,x))_{s\in[0,T],x\in O} \in C^{ 1,2 }( [0,T]\times O, [0,\infty) ) $,
  let $\alpha\in\R$, $\beta\in[0,\infty)$,
  and assume that
  for all $(s,x)\in\cup_{\omega\in\Omega}\{(t,X_t(\omega))\in[0,T]\times \mathcal{O}\colon t\in[0,\tau(\omega)]\}$
  it holds that
  \begin{equation}  \begin{split}\label{eq:condition.exp.moment.condition}
   (\tfrac{\partial}{\partial s}U)(s,x)
    +
    (\tfrac{\partial}{\partial x}U)(s,x)\,
    \mu(s,x)
    &+\tfrac{1}{2}\textup{trace}\Big(
    (\sigma\cdot\sigma^*)(s,x)\,
    (\textup{Hess}_xU)(s,x)
    \Big)
    \\&
    +
    \tfrac{1}{2e^{\alpha s}}\big\|(\tfrac{\partial}{\partial x}U)(s,x)\sigma(s,x)\big\|_{\HS(U,\R)}^2
    \leq \alpha U(s,x)+\beta.
  \end{split}     \end{equation}
  Then
\begin{enumerate}[(i)]
\item
  it holds for all $p\in[1,\infty)$ that
  \begin{equation}  \begin{split}\label{eq:moment.exp.marginal.condition}
    &\E\Big[|p+e^{-\alpha \tau}U(\tau,X_{\tau})|^p\Big]
    \leq
    \E\Big[|p+U(0,X_{0})+\smallint_0^T \tfrac{\beta}{e^{\alpha s}}\,ds|^p\Big]
    \leq
    p^p\E\Big[\exp(U(0,X_{0})+\smallint_0^T \tfrac{\beta}{e^{\alpha s}}\,ds)\Big]
  \end{split}     \end{equation}
  and
\item
  it holds for all $p\in[1,\infty)$, $q\in(0,p)$ that
  \begin{equation}  \begin{split}\label{eq:moment.exp.uniform.condition}
    &\E\!\left[\sup_{t\in[0,\tau]}|p+e^{-\alpha t}U(t,X_{t})|^q\right]
    \leq
    \E\Big[|p+U(0,X_{0})+\smallint_0^T \tfrac{\beta}{e^{\alpha s}}\,ds|^q\Big]
    \Big(
    \tfrac{p}{q}\smallint_{\frac{p-q}{q}}^{\infty}\tfrac{s^{\frac{q}{p}}}{(s+1)^2}\,ds+1\Big)^p.
  \end{split}     \end{equation}
\end{enumerate}
\end{cor}
\begin{proof}[Proof of Corollary~\ref{c:exp.moments.condition}]
  Throughout this proof 
  let $p\in[1,\infty)$
  and
  let
  $V\colon[0,T]\times O\to[0,\infty)$
  satisfy
  for all $s\in[0,T]$, $x\in O$
  that
  $V(s,x)=p+U(s,x)e^{-\alpha s}+\beta\int_s^T e^{-\alpha u}\,du$.
  Note that \eqref{eq:condition.exp.moment.condition} implies
  that
  for all $s\in[0,T]$, $x\in O$
  it holds that
  \begin{equation}  \begin{split}\label{eq:ass.powers.condition}
    &(\tfrac{\partial}{\partial t}V)(s,x)
    +
    (\tfrac{\partial}{\partial x}V)(s,x)\,
    \mu(s,x)
    +\tfrac{1}{2}\textup{trace}\Big(
    (\sigma\cdot\sigma^*)(s,x)\,
    (\textup{Hess}_xV)(s,x)
    \Big)
    \\&\qquad
    +\tfrac{p-1}{2}\tfrac{\|
    (\frac{\partial}{\partial x}V)(s,x)\,\sigma(s,x)\|_{\HS(U,\R)}^2}
    {V(s,x)}
    \\&
    =
   (\tfrac{\partial}{\partial s}U)(s,x)e^{-\alpha s}
   -U(s,x)\alpha e^{-\alpha s}-\beta e^{-\alpha s}
    +
    (\tfrac{\partial}{\partial x}U)(s,x)\,
    \mu(s,x)e^{-\alpha s}
    \\&\qquad
    +\tfrac{1}{2}\textup{trace}\Big(
    (\sigma\cdot\sigma^*)(s,x)\,
    (\textup{Hess}_xU)(s,x)
    \Big)e^{-\alpha s}
    +
    \tfrac{(p-1)}{2}\tfrac{\big\|(\frac{\partial}{\partial x}U)(s,x)\sigma(s,x)e^{-\alpha s}\big\|_{\HS(U,\R)}^2}{p+U(s,x)e^{-\alpha s}+\beta\int_s^{T}e^{-\alpha u}du}
    \\&
    \leq
   e^{-\alpha s}\Big((\tfrac{\partial}{\partial s}U)(s,x)
   -U(s,x)\alpha -\beta
    +
    (\tfrac{\partial}{\partial x}U)(s,x)\,
    \mu(s,x)
    \\&\qquad
    +\tfrac{1}{2}\textup{trace}\Big(
    (\sigma\cdot\sigma^*)(s,x)\,
    (\textup{Hess}_xU)(s,x)
    \Big)
    +
    \tfrac{1}{2e^{\alpha s}}\big\|(\tfrac{\partial}{\partial x}U)(s,x)\sigma(s,x)\big\|_{\HS(U,\R)}^2
    \Big)
    \leq 0.
  \end{split}     \end{equation}
  This, Theorem~\ref{thm:moments:hilbert.m}
  (applied with $a_t=\mu(t,X_t)$, $b_t=\sigma(t,X_t)$,
  $\alpha_t=0$, $\beta_t=0$, $q_2=\infty$ for all $t\in[0,T]$
  in the notation of Theorem~\ref{thm:moments:hilbert.m}),
  and the fact that $\forall x\in[0,\infty)\colon 1+x\leq e^{x}$
  yield that
  \begin{equation}  \begin{split}
    &\E\Big[|p+e^{-\alpha \tau}U(\tau,X_{\tau})|^p\Big]
    \leq
    \E\Big[|V(\tau,X_{\tau})|^p\Big]
    \leq
    \E\Big[|V(0,X_{0})|^p\Big]
    \\&
    =
    \E\Big[|p+U(0,X_{0})+\smallint_0^T\tfrac{\beta}{e^{\alpha s}}\,ds|^p\Big]
    =
    p^p \E\Big[|1+\tfrac{1}{p}U(0,X_{0})+\tfrac{1}{p}\smallint_0^T\tfrac{\beta}{e^{\alpha s}}\,ds|^p\Big]
    \\&
    \leq
    p^p\E\Big[\Big|\exp\Big(\tfrac{1}{p}U(0,X_{0})+\tfrac{1}{p}\smallint_0^T\tfrac{\beta}{e^{\alpha s}}\,ds\Big)\Big|^p\Big]
    =
    p^p\E\Big[\exp(U(0,X_{0})+\smallint_0^T\tfrac{\beta}{e^{\alpha s}}\,ds)\Big]
  \end{split}     \end{equation}
  and yield for all $q\in(0,p)$ that
  \begin{equation}  \begin{split}
    &\left(\E\!\left[\sup_{t\in[0,\tau]}|p+e^{-\alpha t}U(t,X_{t})|^q\right]\right)^{\!\frac{1}{q}}
    \leq
    \left(\E\!\left[\sup_{t\in[0,\tau]}|V(t,X_{t})|^q\right]\right)^{\!\frac{1}{q}}
    \\&
    \leq
    \left(\E\Big[|V(0,X_{0})|^q\Big]\right)^{\!\frac{1}{q}}
    \left(\tfrac{p}{q}\smallint_{\frac{p-q}{q}}^{\infty}\tfrac{s^{\frac{q}{p}}}{(s+1)^2}\,ds+1\right)^{\!\frac{p}{q}}
    \\&=
    \left(\E\Big[|p+U(0,X_{0})+\smallint_0^T \tfrac{\beta}{e^{\alpha u}}\,du|^q\Big]\right)^{\frac{1}q}
    \Big(
    \tfrac{p}{q}\smallint_{\frac{p-q}{q}}^{\infty}\tfrac{s^{\frac{q}{p}}}{(s+1)^2}\,ds+1\Big)^{\frac{p}{q}}.
  \end{split}     \end{equation}
  This finishes the proof of Corollary~\ref{c:exp.moments.condition}.
\end{proof}

\subsection{Strong local Lipschitz continuity in the initial value}\label{sec:localLip}

In this subsection we derive strong local Lipschitz continuity in the initial value of solutions of SDEs.
We do not assume that
the coefficients of the SDE satisfy a global monotonicity condition
since this condition is not satisfied for most example SDEs from applications; cf., e.g.,
Cox et al.~\cite[Chapters 4,5]{CoxHutzenthalerJentzen2013}.
Establishing such a strong local Lipschitz continuity is nontrivial since there exist even SDEs with
globally bounded and smooth coefficients which do not have this property due to a loss of regularity phenomenon; see
Hairer et al.~\cite{HairerHutzenthalerJentzen2015}.
The marginal local Lipschitz estimate~\eqref{eq:l.Lipschitz.marginal} below
improves existing results
in~\cite{CoxHutzenthalerJentzen2013,
FangImkellerZhang2007,
Li1994,%
ScheutzowSchulze2017,%
Zhang2010%
}.
To the best of our knowledge,
the uniform local Lipschitz estimate~\eqref{eq:l.Lipschitz.uniform} below is new.

\begin{sett} \label{s:exists:C0}
  Assume Setting~\ref{sett:applications},
let $ Y \colon [0,T] \times \Omega \to \mathcal{O} $
be an adapted stochastic process
with continuous sample paths
which satisfies that $\P$-a.s.\ it holds that
  $
    \smallint_0^{ \tau } \| \mu( s, Y_s ) \|_H
    + \| \sigma( s, Y_s ) \|_{\HS(U,H)}^2
    \, ds < \infty
  $
and
which satisfies that for all $t\in [0,T]$ it holds $\P$-a.s.~that 
\begin{align} \label{eq:def:Y}
  Y_{\min\{t,\tau\}}& = 
  Y_0
  + \int_0^{ t } \1_{[0,\tau]}(r)\mu(r,Y_{r} ) \, dr
  +
  \int_0^{ t } \1_{[0,\tau]}(r)\sigma(r,Y_{r} ) \, dW_r,
\end{align}
let $ \alpha_0,\alpha_1, \beta_0,\beta_1\in [0,\infty)$,
$ 
  V_0 
$,
$
  V_1  \in C^{ 2 }( O , [0,\infty) ) 
$,
let
$ 
  \bar{V} \colon [0,T] \times \mathcal{O} \to [0,\infty) 
$
be a Borel measurable function which satisfies that
$\P$-a.s.~it holds that
$
	\int_0^{\tau} | \bar{V}(r, X_{r}) |+| \bar{V}(r, Y_{r}) | d r <\infty
$
and that
for all
$ i \in \{ 0, 1 \} $,
$t\in[0,T]$,
$x\in \mathcal{O}$
it holds that 
\begin{equation}
\label{eq:multiple_exp_est3}
\begin{split}
&  \langle \mu(t,x),(\nabla V_i)(x)\rangle_H+\tfrac{1}{2}\textup{trace}\Big(\sigma(t,x)[\sigma(t,x)]^{*}(\textup{Hess}V_i)(x)\Big)
\\&
  +
  \tfrac{ 
    1
  }{ 
    2 
    e^{ 
      \alpha_{ i } t 
    }
  }
    \|
      \sigma( t, x )^* ( \nabla V_{ i } )( x )
    \|_U^2
  +
  \mathbbm{1}_{
    \{ 1 \}
  }(i)
  \cdot
  \bar{V}(t,x)
\leq
  \alpha_{ i } V_{ i }(x)
  +
  \beta_{ i },
\end{split}
\end{equation}
let 
$
  \constFun \colon [0,T] \to [0,\infty]
$
be a Borel measurable function which satisfies that $\int_0^T \constFun(r) \, dr <\infty$,
let $p\in[2,\infty)$, $q,q_0,q_1\in(0,\infty)$ satisfy that $\tfrac{1}{q_0}+\tfrac{1}{q_1}=\tfrac{1}{q}$,
and assume that for all $t\in[0,T]$, $x,y\in \mathcal{O}$ it holds that
\begin{equation}  \begin{split}\label{eq:ass.41}
  &\big\langle x-y,\mu(t,x)-\mu(t,y)\big\rangle_H+\tfrac{1}{2}\big\|\sigma(t,x)-\sigma(t,y)\big\|^2_{\HS(U,H)}
  +\tfrac{p-2}{2}
    \tfrac{\big\|\big\langle x-y,\sigma(t,x)-\sigma(t,y)\big\rangle_H\big\|_{\HS(U,\R)}^2}{\|x-y\|_H^2}
  \\&
  \leq \|x-y\|^2_H\cdot\Big(\constFun(t)
  +\tfrac{V_0(x)+V_0(y)}{2q_0 T e^{\alpha_0t}}
  +\tfrac{\bar{V}(t,x)+\bar{V}(t,y)}{2q_1 e^{\alpha_1t}}
  \Big).
\end{split}     \end{equation}
\end{sett}

\begin{lemma}[Strong local Lipschitz continuity in the initial value] \label{l:localLipschitz.initial}
Assume Setting~\ref{s:exists:C0} and let $x,y\in \mathcal{O}$.
Then
\begin{enumerate}[(i)]
  \item\label{item:localLipschitz1}  
it holds for all $t\in(0,T]$ that
  \begin{equation}  \begin{split}\label{eq:l.Lipschitz.marginal}
    \Big\|X_{\min\{t,\tau\}}-Y_{\min\{t,\tau\}}\Big\|_{L^{\frac{pq}{p+q}}(\P;H)}
&\leq \Big\|X_0-Y_0\Big\|_{L^p(\P;H)}
  \exp\!\left(
        \int_0^t  
        \constFun(r)
        +
        \tfrac{
          \beta_{ 0 } \, ( 1 - \frac{ r }{ t } )
        }{
          q_{0} e^{ \alpha_{ 0 } r }
        }
        + 
        \tfrac{ \beta_{ 1 } }{ 	q_{ 1 } e^{ \alpha_{ 1 } r } }
        \,
        dr
      \right)
  \\&\qquad\cdot
  \prod_{i=0}^1
  \left(\E\left[\exp\left(V_i(X_0)\right)\right]\right)^{\!\nicefrac{1}{2q_i}}
  \prod_{i=0}^1
  \left(\E\left[\exp\left(V_i(Y_0)\right)\right]\right)^{\!\nicefrac{1}{2q_i}}
  \end{split}     \end{equation}
  and
  \item\label{item:localLipschitz2}  
it holds for all $\delta\in(0,1)$ that
  \begin{equation}  \begin{split}\label{eq:l.Lipschitz.uniform}
 &   \Big\|\sup_{t\in[0,\tau]}\|X_t-Y_t\|_H\Big\|_{L^{\frac{pq\delta}{p\delta+q}}(\P;\R)}
\leq
\Big\|X_0-Y_0\Big\|_{L^{p\delta}(\P;H)}
  \exp\!\left(
        \int_0^T  
        \constFun(r)
        +
        \tfrac{
          \beta_{ 0 } 
          \, ( 1 - \frac{ r }{ T } )
        }{
          q_{0} e^{ \alpha_{ 0 } r }
        }
        + 
        \tfrac{ \beta_{ 1 } }{ 	q_{ 1 } e^{ \alpha_{ 1 } r } }
        \,
        dr
      \right)
  \\&\qquad\qquad\cdot
     \left(\tfrac{1}{\delta}\smallint_{\frac{1-\delta}{\delta}}^{\infty}\tfrac{s^{\delta}}{(s+1)^2}\,ds+1\right)^{\!\frac{1}{2\delta }}
  \prod_{i=0}^1
  \Big(\E\left[\exp\left(V_i(X_0)\right)\right]\Big)^{\nicefrac{1}{2q_i}}
  \prod_{i=0}^1
  \Big(\E\left[\exp\left(V_i(Y_0)\right)\right]\Big)^{\nicefrac{1}{2q_i}}
  \! .
  \end{split}     \end{equation}
\end{enumerate}
\end{lemma}
\begin{proof}[Proof of Lemma~\ref{l:localLipschitz.initial}]
  It follows from~\eqref{eq:appl.X} and \eqref{eq:def:Y}
  that for all $t\in[0,T]$ it holds $\P$-a.s.\ that
  \begin{equation}  \begin{split}\label{eq:ItoXZ}
    &X_{\min\{t,\tau\}}- Y_{\min\{t,\tau\}}
    \\&=x-y+\int_0^t\1_{[0,\tau]}(r)\left( \mu(r,X_r)-\mu(r,Y_r)\right)\,dr
    +\int_0^t \1_{[0,\tau]}(r)\left(\sigma(r,X_r)-\sigma(r,Y_r)\right)\,dW_r.
  \end{split}     \end{equation}
  Assumption~\eqref{eq:ass.41} implies for all $t\in[0,T]$ that
  \begin{equation}  \begin{split}\label{eq:XZ.OSALG}
    &\langle X_t-Y_t,\mu(t,X_t)-\mu(t,Y_t)\rangle_H
    +\tfrac{1}{2}\Big\|\sigma(t,X_t)-\sigma(t,Y_t)\Big\|_{HS(U,H)}^2
    \\&\quad
    +\tfrac{p-2}{2}
     \tfrac{\|\langle X_t-Y_t,\sigma(t,X_t)-\sigma(t,Y_t)\rangle_{H}\|_{\HS(U,\R)}^2}
           {\| X_t-Y_t\|_H^2}
    \\&
    \leq \Big\|X_t-Y_t\Big\|_H^2\cdot
    \left(\constFun(t)+\tfrac{V_0(X_t)+V_0(Y_t)}{2q_0Te^{\alpha_0t}}
    +\tfrac{\bar{V}(t,X_t)+\bar{V}(t,Y_t)}{2q_1e^{\alpha_1t}}
    \right).
  \end{split}     \end{equation}
  This, \eqref{eq:ItoXZ}, and item (i) in Corollary~\ref{c:moments:hilbert.m}
  (applied for every $s\in(0,T]$ with 
  $T=s$, $\tau=\min\{s,\tau\}$, $X_t=X_t-Y_t$,
  $a_t=\mu(t,X_t)-\mu(t,Y_t)$,
  $b_t=\sigma(t,X_t)-\sigma(t,Y_t)$,
  $\alpha_t=
    \constFun(t)+\tfrac{V_0(X_t)+V_0(Y_t)}{2q_0Te^{\alpha_0t}}
    +\tfrac{\bar{V}(t,X_t)+\bar{V}(t,Y_t)}{2q_1e^{\alpha_1t}}$,
  $\beta_t=0$,
  $q_1=\tfrac{pq}{p+q}$,
  $q_2=q$
  for all $t\in[0,s]$
  in the notation of Corollary~\ref{c:moments:hilbert.m})
  imply for all $t\in(0,T]$ that
  \begin{equation}  \begin{split}\label{eq:local.Lipschitz.proof}
    &\Big\|X_{\min\{t,\tau\}}-Y_{\min\{t,\tau\}}\Big\|_{L^{\frac{pq}{p+q}}(\P;H)}
\\&
\leq \Big\|X_0-Y_0\Big\|_{L^p(\P;H)}
    \left\|
      \exp\!\left(
 \int_{0}^{{\min\{t,\tau\}}}
  \constFun(r)
  +
    \tfrac{ 
      V_{ 0 }( X_r ) 
      +
      V_{ 0 }( Y_r ) 
    }{ 2 q_{ 0 } T e^{ \alpha_{ 0} r } } 
    +
    \tfrac{ 
     \bar{V}( r, X_r ) 
     +
     \bar{V}( r, Y_r ) 
    }{ 
      2 q_{ 1} 
      e^{ \alpha_{ 1} r }
    }
  \,dr
  \right)
    \right\|_{
      L^q( \P; \R )
    }
\\&
\leq \Big\|X_0-Y_0\Big\|_{L^p(\P;H)}
    \left\|
      \exp\!\left(
 \int_{0}^{{\min\{t,\tau\}}}
  \constFun(r)
  +
    \tfrac{ 
      V_{ 0 }( X_r ) 
      +
      V_{ 0 }( Y_r ) 
    }{ 2 q_{ 0 } t e^{ \alpha_{ 0} r } } 
    +
    \tfrac{ 
     \bar{V}( r, X_r ) 
     +
     \bar{V}( r, Y_r ) 
    }{ 
      2 q_{ 1} 
      e^{ \alpha_{ 1} r }
    }
  \,dr
  \right)
    \right\|_{
      L^q( \P; \R )
    }.
  \end{split}     \end{equation}
H{\"o}lder's inequality together with
$\tfrac{ 1 }{ q } =2\tfrac{ 1 }{ 2q_{ 0} }+2\tfrac{ 1 }{ 2q_{ 1} }$,
the fact that $\beta_0,\beta_1\geq0$,
the fact that
$\forall t\in(0,T]\colon
\int_0^{t} \tfrac{\beta_0(1-\frac{r}{t})}{q_0 e^{\alpha_0 r }} \, dr
=
\int_0^{t} \int_0^r \tfrac{\beta_0}{q_0te^{\alpha_0 u}} \, du \, dr
\geq 
\int_0^{\min\{t,\tau\}} \int_0^r \tfrac{\beta_0}{q_0te^{\alpha_0 u}} \, du \, dr
$,
Jensen's inequality,
Tonelli's theorem,
nonnegativity of $V_1$,
\eqref{eq:multiple_exp_est3},
and
Corollary \ref{c:exp.moments}
(applied for every $r\in(0,T]$ with 
$\tau = \min\{r,\tau\}$,
$U(s,x)=V_0(x)$,
$\bar{U}(s,x)= -\beta_0$,
$\alpha=\alpha_0$,
$X = X$ (resp.\ $X=Y$) 
for all $s\in (0,T]$, $x\in\mathcal{O}$
and applied for every $t\in(0,T]$ with
$\tau=\min\{t,\tau\}$,
$U(s,x)=V_1(x)$,
$\bar{U}(s,x)=\bar{V}(s,x)-\beta_1$,
$\alpha=\alpha_1$,
$X = X$ (resp.\ $X=Y$)
for all $s\in(0,T]$, $x\in\mathcal{O}$
in the notation of
Corollary \ref{c:exp.moments.condition})
show for all $t\in(0,T]$ that
\begin{equation}\label{eq:estimate.exp}
\begin{split}
&
    \left\|
      \exp\!\left(
 \int_0^{\min\{t,\tau\}}
    \tfrac{ 
      V_{ 0 }( X_r ) 
      +
      V_{ 0 }( Y_r ) 
    }{ 2 q_{ 0 } t e^{ \alpha_{ 0} r } } 
    +
    \tfrac{ 
     \bar{V}(r, X_r ) 
     +
     \bar{V}(r, Y_r ) 
    }{ 
      2 q_{ 1} 
      e^{ \alpha_{ 1} r }
    }
  \,dr
  \right)
    \right\|_{
      L^q( \P; \R )
    }
  \exp\left(
    -\int_0^{t}\sum_{i=0}^1\tfrac{\beta_i(1-\frac{r}{t})^{1-i}}{q_ie^{\alpha_ir}}\,dr\right)
    \\
    &
\leq
    \left\|
      \exp\!\left(
 \int_0^{{\min\{t,\tau\}}}
    \tfrac{ 
      V_{ 0 }( X_r ) 
    }{  2q_0 t e^{ \alpha_{ 0} r } } 
    -\int_0^r\tfrac{\beta_0}{2 q_0 t e^{\alpha_0 u}}\,du
    \,dr
    \right)
    \right\|_{
      L^{2q_0}( \P; \R )
    }
 \left\|
   \exp\!\left(
   \int_0^{{\min\{t,\tau\}}} \!\!
    \tfrac{ 
     \bar{V}(r, X_r )-\beta_1 
    }{ 
      2 q_1 e^{ \alpha_{ 1} r }
    }
  \,dr
  \right)
    \right\|_{
      L^{2q_1}( \P; \R )
    }
 \\&\quad\cdot
    \left\|
      \exp\!\left(
 \int_0^{{\min\{t,\tau\}}}
    \tfrac{ 
      V_{ 0 }( Y_r ) 
    }{  2q_0 t e^{ \alpha_{ 0} r } } 
    -\int_0^r\tfrac{\beta_0}{2 q_0 t e^{\alpha_0 u}}\,du
    \,dr
    \right)
    \right\|_{
      L^{2q_0}( \P; \R )
    }
 \left\|
   \exp\!\left(
   \int_0^{{\min\{t,\tau\}}} \!\!
    \tfrac{ 
     \bar{V}(r, Y_r )-\beta_1 
    }{ 
      2 q_1 e^{ \alpha_{ 1} r }
    }
  \,dr
  \right)
    \right\|_{
      L^{2q_1}( \P; \R )
    }
    \\
    &
\leq
\left(
\tfrac{1}{t}\int_0^t
    \E\!\left[
      \exp\!\left(
    \tfrac{ 
      V_{ 0 }( X_{\min\{r,\tau\}} ) 
    }{   e^{ \alpha_{ 0} {\min\{r,\tau\}} } } 
    -\int_0^{\min\{r,\tau\}} \tfrac{\beta_0}{e^{\alpha_0 u}}\,du
    \right)
    \right]
    \,dr
    \right)^{\!\!\nicefrac{1}{2q_0}}
 \\ & \quad \cdot 
   \left(\E\!\left[
   \exp\!\left(
    \tfrac{ 
       V_{ 1 }( X_{\min\{t,\tau\}} ) 
    }{ e^{ \alpha_{ 1} {\min\{t,\tau\}} } } 
    +
   \int_0^{{\min\{t,\tau\}}}
    \tfrac{ 
     \bar{V}(r, X_r ) -\beta_1
    }{ 
      e^{ \alpha_{ 1} r } 
    }
  \,dr
  \right)
  \right]\right)^{\!\!\nicefrac{1}{2q_1}}
    \\
    &
\quad\cdot
\left(
\tfrac{1}{t}\int_0^t
    \E\!\left[
      \exp\!\left(
    \tfrac{ 
      V_{ 0 }( Y_{\min\{r,\tau\}} ) 
    }{   e^{ \alpha_{ 0} {\min\{r,\tau\}} } } 
    -\int_0^{\min\{r,\tau\}} \tfrac{\beta_0}{e^{\alpha_0 u}}\,du
    \right)
    \right]
    \right)^{\!\!\nicefrac{1}{2q_0}}
 \\ & \quad \cdot 
   \left(\E\!\left[
   \exp\!\left(
    \tfrac{ 
       V_{ 1 }( Y_{\min\{t,\tau\}} ) 
    }{ e^{ \alpha_{ 1} {\min\{t,\tau\}} } } 
    +
   \int_0^{{\min\{t,\tau\}}}
    \tfrac{ 
     \bar{V}(r, Y_r ) -\beta_1
    }{ 
      e^{ \alpha_{ 1} r } 
    }
  \,dr
  \right)
  \right]\right)^{\!\!\nicefrac{1}{2q_1}}
\\ &
\leq
  \prod_{i=0}^1
  \left(\E\left[\exp\left(V_i(X_0)\right)\right]\right)^{\!\nicefrac{1}{2q_i}}
  \prod_{i=0}^1
  \left(\E\left[\exp\left(V_i(Y_0)\right)\right]\right)^{\!\nicefrac{1}{2q_i}}
  \!\! .
\end{split}
\end{equation}
This and inequality~\eqref{eq:local.Lipschitz.proof}
yield for all $t\in(0,T]$ that
  \begin{equation}  \begin{split}
    \Big\|X_{\min\{t,\tau\}}-Y_{\min\{t,\tau\}}\Big\|_{L^{\frac{pq}{p+q}}(\P;H)}
&\leq \Big\|X_0-Y_0\Big\|_{L^p(\P;H)}
  \exp\!\left(
        \int_0^t  
        \constFun(r)
        +
        \tfrac{
          \beta_{ 0 }
          \, ( 1 - \frac{ r }{ t } )
        }{
          q_{0} e^{ \alpha_{ 0 } r }
        }
        + 
        \tfrac{ \beta_{ 1 } }{ 	q_{ 1 } e^{ \alpha_{ 1 } r } }
        \,
        dr
      \right)
  \\&\qquad\cdot
  \prod_{i=0}^1
  \left(\E\left[\exp\left(V_i(X_0)\right)\right]\right)^{\!\nicefrac{1}{2q_i}}
  \prod_{i=0}^1
  \left(\E\left[\exp\left(V_i(Y_0)\right)\right]\right)^{\!\nicefrac{1}{2q_i}}
  \! .
  \end{split}     \end{equation}
  This proves item~\eqref{item:localLipschitz1}.
  Next, \eqref{eq:ItoXZ}, \eqref{eq:XZ.OSALG},
  item (ii) in Corollary~\ref{c:moments:hilbert.m}
  (applied for every $\delta\in(0,1)$ with 
  $\tau=\tau$,
  $X_t=X_t-Y_t$,
  $a_t=\mu(t,X_t)-\mu(t,Y_t)$,
  $b_t=\sigma(t,X_t)-\sigma(t,Y_t)$,
  $\alpha_t=
    \constFun(t)+\tfrac{V_0(X_t)+V_0(Y_t)}{2q_0Te^{\alpha_0t}}
    +\tfrac{\bar{V}(t,X_t)+\bar{V}(t,Y_t)}{2q_1e^{\alpha_1t}}$,
  $\beta_t=0$,
  $q_1=\tfrac{pq\delta}{p\delta+q}$,
  $q_2=q$,
  $q_3=\delta p$
  for all $t\in[0,T]$
  in the notation of Corollary~\ref{c:moments:hilbert.m}),
  and \eqref{eq:estimate.exp}
  imply for all $\delta\in(0,1)$ that
  \begin{equation}  \begin{split}
 &   \Big\|\sup_{t\in[0,\tau]}\|X_t-Y_t\|_H\Big\|_{L^{\frac{pq\delta}{p\delta+q}}(\P;\R)}
\leq
  \left(\tfrac{1}{\delta}\smallint_{\frac{1-\delta}{\delta}}^{\infty}\tfrac{s^{\delta}}{(s+1)^2}\,ds+1\right)^{\!\frac{1}{2\delta }}
\Big\|X_0-Y_0\Big\|_{L^{p\delta}(\P;H)}
\\&\qquad\qquad\cdot
    \left\|
      \exp\!\left(
 \int_{0}^{T}
  \constFun(r)
  +
    \tfrac{ 
      V_{ 0 }( X_r ) 
      +
      V_{ 0 }( Y_r ) 
    }{ 2 q_{ 0 } T e^{ \alpha_{ 0} r } } 
    +
    \tfrac{ 
     \bar{V}( r, X_r ) 
     +
     \bar{V}( r, Y_r ) 
    }{ 
      2 q_{ 1} 
      e^{ \alpha_{ 1} r }
    }
  \,dr
  \right)
    \right\|_{
      L^q( \P; \R )
    }
\\&
\leq
     \left(\tfrac{1}{\delta}\smallint_{\frac{1-\delta}{\delta}}^{\infty}\tfrac{s^{\delta}}{(s+1)^2}\,ds+1\right)^{\!\frac{1}{2\delta }}
\Big\|X_0-Y_0\Big\|_{L^{p\delta}(\P;H)}
  \exp\!\left(
        \int_0^T  
        \constFun(r)
        +
        \tfrac{
          \beta_{ 0 } \, ( 1 - \frac{ r }{ T } )
        }{
          q_{0} e^{ \alpha_{ 0 } r }
        }
        + 
        \tfrac{ \beta_{ 1 } }{ 	q_{ 1 } e^{ \alpha_{ 1 } r } }
        \,
        dr
      \right)
  \\&\qquad\qquad\cdot
  \prod_{i=0}^1
  \left(\E\left[\exp\left(V_i(X_0)\right)\right]\right)^{\!\nicefrac{1}{2q_i}}
  \prod_{i=0}^1
  \left(\E\left[\exp\left(V_i(Y_0)\right)\right]\right)^{\!\nicefrac{1}{2q_i}}
  \! .
  \end{split}     \end{equation}
  This proves item~\eqref{item:localLipschitz2} and completes the proof of Lemma~\ref{l:localLipschitz.initial}.
\end{proof}

The following lemma, Lemma~\ref{l:temp.reg}, is essentially well-known
and is included for the convenience of the reader.

\begin{lemma}[Temporal regularity] \label{l:temp.reg}
Assume Setting~\ref{s:exists:C0},
let
$\gamma\in\big[\frac{1}{p},\infty\big)$, $c\in[0,\infty)$
satisfy
for all $t\in[0,T]$,
$ x\in \mathcal{O}$
that
\begin{align}\label{eq:growth.coeff}
\begin{split}
 \max\left\{ \|\mu(t,x)\|_H , \|\sigma(t,x)\|_{\HS(U,H)} \right\}
 \leq  c(1+V_0(x))^{\gamma},
\end{split}
\end{align}
and let 
$s\in[0,T]$.
%
%
Then it holds that
\begin{align} \label{lem:moments:X:eq2}
\begin{split}
& \left\| \sup_{t\in[\min\{s,\tau\},\tau]}\|X_{t} -X_{\min\{s,\tau\}}\|_H \right\|_{L^{p}(\P;\R)} 
\\&
  \leq c
    e^{\alpha_0\gamma T}\left\|p\gamma+V_0(X_0)+\int_0^T\tfrac{\beta_0}{e^{\alpha_0 u}}\,du\right\|_{L^{p\gamma}(\P;\R)}^{\gamma}
    \big(\sqrt{T}+p\big)\sqrt{|T-s|}.
\end{split}
\end{align} 
\end{lemma}
\begin{proof}[Proof of Lemma~\ref{l:temp.reg}]
Equation \eqref{eq:appl.X},
the triangle inequality,
the Burkholder-Davis-Gundy type inequality in Da Prato \& Zabczyk \cite[Lemma 7.2 and Lemma 7.7]{dz92}, 
\eqref{eq:growth.coeff},
the fact that $p\gamma\geq 1$,
the fact that $\beta_0\geq 0$,
and
Corollary~\ref{c:exp.moments.condition}
(applied for every $r\in[s,T]$ with $U(t,x)=V_0(x)$,
$\alpha=\alpha_0$,
$\beta=\beta_0$
for all $t\in[s,T]$, $x\in O$
in the notation of 
Corollary~\ref{c:exp.moments.condition})
ensure that
\begin{align}
\begin{split}
& \left\| \sup_{t\in[\min\{s,\tau\},\tau]}\|X_{t} -X_{s}\|_H \right\|_{L^{p}(\P;\R)} 
\\&
	\leq
   \left\| \sup_{t\in[s,T]}\left\|\smallint_{s}^{t}\1_{[0,\tau]}(r) 	\mu(r,X_{r} ) \, dr\right\|_H \right\|_{L^{p}(\P;\R)} 
	+  \left\|\sup_{t\in[s,T]}\left\|\smallint_{s}^{t}\1_{[0,\tau]}(r) \sigma(r,X_{r} ) \, dW_r\right\|_H \right\|_{L^{p}(\P;\R)} 
\\&
	\leq
    \int_{s}^{T} \left\|\1_{[0,\tau]}(r)	\mu(r,X_{r} ) \right\|_{L^{p}(\P;H)}  \, dr
	+  \left( \tfrac{p^3}{2(p-1)}\smallint_{s}^{T}\Big\|\1_{[0,\tau]}(r) \sigma(r,X_{r} ) \Big\|_{L^{p}(\P;\HS(U,H))}^2\,dr\right)^{\!\frac{1}{2}}
\\&
	\leq
    c\smallint_{s}^{T}\left(\E\Big[\big(1+V_0(X_{\min\{r,\tau\}} )\big)^{p\gamma}\Big] \right)^{\!\frac{1}{p}}  \, dr
	+ c \left(\tfrac{p^3}{2(p-1)} \smallint_{s}^{T}
  \left(\E\Big[\big(1+V_0(X_{\min\{r,\tau\}} )\big)^{p\gamma}\Big] \right)^{\!\frac{2}{p}}\,dr\right)^{\!\frac{1}{2}}
\\&
	\leq
    c\sup_{r\in[s,T]}\left(e^{\alpha_0\gamma pT}
    \E\Big[\big(p\gamma+e^{-\alpha_0 \min\{r,\gamma\}}V_0(X_{\min\{r,\tau\}} )\big)^{p\gamma}\Big] \right)^{\!\frac{1}{p}}
    \left(T-s+\sqrt{(T-s) \tfrac{p^3}{2(p-1)}} \right)
\\&
	\leq
   c \Big(e^{\alpha_0 p\gamma T}\E\Big[|p\gamma+V_0(X_0)+\int_0^T\tfrac{\beta_0}{e^{\alpha_0u}}\,du|^{p\gamma}\Big]
    \Big)^{\frac{1}{p}}
    \Big(T-s+\sqrt{(T-s)\tfrac{p^3}{2(p-1)}}\Big)
 \\&
  \leq
   c e^{\alpha_0\gamma T}\Big\|p\gamma+V_0(X_0)+\int_0^T\tfrac{\beta_0}{e^{\alpha_0u}}\,du\Big\|_{L^{p\gamma}(\P;\R)}^{\gamma}
    \big(\sqrt{T}+p\big)\sqrt{|T-s|}.
\end{split}
\end{align}
The proof of Lemma~\ref{l:temp.reg} is thus completed.
\end{proof}

\begin{lemma}[Strong local H\"older estimate] \label{lem:Hoelder}
Assume Setting~\ref{s:exists:C0},
let
$\gamma\in[\tfrac{1}{p},\infty)$, $c\in[0,\infty)$
satisfy
for all $t\in[0,T]$,
$ x\in \mathcal{O}$
that
\begin{align}
\begin{split}
 \max\left\{ \|\mu(t,x)\|_H , \|\sigma(t,x)\|_{\HS(U,H)} \right\}
 \leq  c(1+V_0(x))^{\gamma},
\end{split}
\end{align}
assume that $\tfrac{pq}{p+q}\in[2,\infty)\cap[\tfrac{1}{\gamma},\infty)$
and
let
$t_1,t_2\in [0,T]$, $x_1,x_2\in \mathcal{O}$.
Then it holds that 
\begin{equation}  \begin{split}
  \|X_{t_1}-Y_{t_2}\|_{L^{\frac{pq}{p+q}}(\P;H)}
&
\leq
    \sqrt{|t_1-t_2|}c e^{\alpha_0\gamma T}
    \left\|\tfrac{pq}{p+q}\gamma+V_0(X_0)+\int_0^T\tfrac{\beta_0}{e^{\alpha_0 s}}\,ds\right\|_{L^{p\gamma}(\P;\R)}^{\gamma}
    \big(\sqrt{T}+\tfrac{pq}{p+q}\big)
\\&\qquad
+
 \Big\|X_0-Y_0\Big\|_{L^p(\P;H)}
  \exp\!\left(
        \int_0^{T}
        \constFun(r)
        +
        \tfrac{
          \beta_{ 0 }
        }{
          q_{0} e^{ \alpha_{ 0 } r }
        }
        + 
        \tfrac{ \beta_{ 1 } }{ 	q_{ 1 } e^{ \alpha_{ 1 } r } }
        \,
        dr
      \right)
  \\&\qquad\cdot
  \prod_{i=0}^1
  \left(\E\left[\exp\left(V_i(X_0)\right)\right]\right)^{\!\nicefrac{1}{2q_i}}
  \prod_{i=0}^1
  \left(\E\left[\exp\left(V_i(Y_0)\right)\right]\right)^{\!\nicefrac{1}{2q_i}}
  \! .
\end{split}     \end{equation}
\end{lemma}
\begin{proof}[Proof of Lemma~\ref{lem:Hoelder}]
Without loss of generality we assume that $t_1+t_2>0$.
The triangle inequality,
Lemma~\ref{l:temp.reg} 
(applied with $T=\max\{t_1,t_2\}$,
$\tau=\max\{t_1,t_2\}$,
$p=\tfrac{pq}{p+q}$,
$s=\min\{t_1,t_2\}$
in the notation of 
Lemma~\ref{l:temp.reg}),
and
Lemma~\ref{l:localLipschitz.initial}
(applied with $\tau=T$, $x=x_1$, $y=x_2$, $t=t_2$
in the notation of
Lemma~\ref{l:localLipschitz.initial})
yield that
\begin{equation}  \begin{split}
  &\|X_{t_1}-Y_{t_2}\|_{L^{\frac{pq}{p+q}}(\P;H)}
  \leq
  \|X_{t_1}-X_{t_2}\|_{L^{\frac{pq}{p+q}}(\P;H)}
  +
  \|X_{t_2}-Y_{t_2}\|_{L^{\frac{pq}{p+q}}(\P;H)}
\\&
\leq
    \sqrt{|t_1-t_2|}c e^{\alpha_0\gamma T}
    \left\|\tfrac{pq}{p+q}\gamma+V_0(X_0)+\int_0^T\tfrac{\beta_0}{e^{\alpha_0 s}}\,ds\right\|_{L^{p\gamma}(\P;\R)}^{\gamma}
    \big(\sqrt{T}+\tfrac{pq}{p+q}\big)
\\&\qquad
+
 \Big\|X_0-Y_0\Big\|_{L^p(\P;H)}
  \exp\!\left(
        \int_0^{T}
        \constFun(r)
        +
        \tfrac{
          \beta_{ 0 }
        }{
          q_{0} e^{ \alpha_{ 0 } r }
        }
        + 
        \tfrac{ \beta_{ 1 } }{ 	q_{ 1 } e^{ \alpha_{ 1 } r } }
        \,
        dr
      \right)
  \\&\qquad\cdot
  \prod_{i=0}^1
  \left(\E\left[\exp\left(V_i(X_0)\right)\right]\right)^{\!\nicefrac{1}{2q_i}}
  \prod_{i=0}^1
  \left(\E\left[\exp\left(V_i(Y_0)\right)\right]\right)^{\!\nicefrac{1}{2q_i}}
  \! .
\end{split}     \end{equation}
This completes the proof of Lemma~\ref{lem:Hoelder}.
\end{proof}

\subsection{Strong completeness}\label{sec:completeness}
In this subsection we derive conditions on the coefficients
of an SDE which ensure strong completeness of the SDE.
For this we first derive
a version of the Kolmogorov-Chentsov theorem.
More precisely,
the following proposition, Proposition~\ref{p:KolChen},
provides a method which allows
to obtain a continuous version
of a mapping $X\colon D\times\Omega\to F$
from a subset $D$ of a finite-dimensional Hilbert space
to a closed subset of a Banach space
if there exist $p\in(\dim(H),\infty)$ and $\alpha\in(\tfrac{\dim(H)}{p},1]$
such that the mapping $D\ni x\mapsto X(x)\in L^p(\P;F)$
is locally bounded and locally $\alpha$-H\"older continuous.
In the case where $F=E$, $H=\R^d$, and~\eqref{eq:locallyHoelder} below holds for $n=\infty$, the proof of Proposition~\ref{p:KolChen}
is provided in Theorem 2.1 in Mittmann \& Steinwart~\cite{MS03}.
Proposition~\ref{p:KolChen} slightly generalizes
Cox et al.~\cite[Theorem 3.5]{CoxHutzenthalerJentzen2013}
and Grohs et al.~\cite[Lemma 2.19]{GrohsHornungJentzenWurstemberger2018}.
\begin{prop}[Existence of a continuous version]\label{p:KolChen}
Let
$(\Omega, \mathcal{F}, \P)$ be a probability space,  
let $( H, \left< \cdot , \cdot \right>_H, \left\| \cdot \right\|_H )$
be a finite-dimensional $\R$-Hilbert space,
let $ D  \subseteq H$ be a set,
let $ ( E, \left\| \cdot \right\|_E )$ be a Banach space,
let $ F \subseteq E $ be a closed subset,
let
$ p \in (\dim(H),\infty) $,  
$ 
  \alpha \in ( \frac{ \dim(H) }{ p } , \infty) 
$,
and let 
$ 
  X \colon D \times\Omega\to F
$
be a random field which satisfies for all $n\in\N$
that 
\begin{equation}  \begin{split}\label{eq:locallyHoelder}
  \sup\Big(&\Big\{\E\left[\|X(x)\|_E^p\right]\colon
  x\in D, \|x\|_H\leq n
  \Big\}
  \\&
  \cup
  \Big\{\tfrac{\left(\E\left[\|X(x)-X(y)\|_E^p\right]\right)^{\!\frac{1}{p}}}
  {\|x-y\|_H^{\alpha }}
  \colon x,y\in D, \|x\|_H\leq n, \|y\|_H\leq n,x\neq y \Big\}
  \cup\{0\}
  \Big)
  <\infty.
\end{split}     \end{equation}
Then there exists 
a function 
$
  \mathcal{X} \colon \overline{D} \times \Omega \to F
$
which satisfies
\begin{enumerate}[(i)]
  \item that 
      $\mathcal{X}$ is
      $\mathcal{B}(\overline{D}) \otimes \mathcal{F}/ \mathcal{B}(F)$-measurable,
  \item 
that for all $\omega\in\Omega$ it holds that
$(\overline{D}\ni x\mapsto \mathcal{X}(x,\omega)\in F) \in C(\overline{D}, F)$,
\item for all $n\in\N$, $\beta\in(0,\alpha-\tfrac{\dim(H)}{p})$
  that
  $\mathcal{X}|_{\{x\in\overline{D}\colon \|x\|_H\leq n\}}
  \in \mathcal{L}^p(\P; C_b^{\beta}(\{x\in\overline{D}\colon \|x\|_H\leq n\},F))$,
\item that for all $x\in D$ it holds $\P$-a.s.~that
$
  \mathcal{X}(x) = X(x)
$.
\end{enumerate}
\end{prop}
\begin{proof}[Proof of Proposition~\ref{p:KolChen}]
Without loss of generality we assume that $D\neq \emptyset$
(otherwise the assertion is trivial)
and that $H\neq \{0\}$
(if $H=\{0\}$, then $D=\{0\}$ and $X$ itself satisfies (i)--(iii)).
Throughout this proof
for every $n\in\N$ let $D_n \subseteq D$
be the set which satisfies that $D_n=\{x\in D\colon\|x\|_H\leq n\}$,
let $d\in\N$ satisfy that $d=\dim(H)$,
let $\{h_1, \ldots, h_d \}\subseteq H$ be an orthonormal basis of $H$, and let $\mathcal{D} \subseteq D$ be the set
$\mathcal{D}=\{x\in D \colon (\langle x, h_i\rangle_H)_{i=1}^d \in \mathbb{Q}^d\}$. 
By assumption it holds for all $n\in\N$ that $X|_{D_n}\in C_b^{\alpha}(D_n,\mathcal{L}^p(\P;F))$.
Then Theorem~3.5 in Cox et al.~\cite{CoxHutzenthalerJentzen2013}
shows that for all $n\in \N$
there exists
$
  \mathcal{X}^n \in  
  \bigcap_{ 
    \beta \in ( 0, \alpha - \frac{ d }{ p } )
  }
  \,
  \mathcal{L}^p(
    \P; 
    C^{ \beta }_b(\overline{D_n}, F )
  )
$
such that 
for every $ x \in D_n $ 
it holds $\P$-a.s.~that
$
 \mathcal{X}^n(x) = X(x)
$.
Let $\Omega_0\subseteq\Omega$ be the set satisfying that
$\Omega_0
=\cap_{n\in\N} \cap_{x\in D_n\cap \mathcal{D}} \cap_{m\in\N\cap[n,\infty)} \{\mathcal{X}^m(x)=X(x)\}$.
Then this, the fact that $X$, $(\mathcal{X}^n)_{n\in\N}$ are random fields, and the fact that
$\N\times\mathcal{D}\times\N$ is a countable set imply that $\Omega_0\in\mathcal{F}$
and
that $\P(\Omega_0)=1$.
Continuity yields that for all $\omega\in\Omega_0$, $n,m\in\N$ with $m\geq n$ it holds that $\mathcal{X}^m(\omega)|_{\overline{D_n}}=\mathcal{X}^n(\omega)$.
Note that $\overline{D}=\cup_{n=1}^{\infty}\overline{D_n}$.
Now let $\mathcal{X}\colon\overline{D}\times\Omega\to F$ be the function satisfying for all $x\in \overline{D}$, $\omega\in\Omega$
that $\mathcal{X}(x,\omega)=\mathbbm{1}_{\Omega_0}(\omega)\limsup_{n\to\infty}\mathcal{X}^n(x,\omega)$.
Then it holds for all $n\in\N$ that $\mathcal{X}|_{\overline{D_n}\times\Omega}=\mathbbm{1}_{\Omega_0}\mathcal{X}^n\in
  \bigcap_{ 
    \beta \in ( 0, \alpha - \frac{ d }{ p } )
  }
  \,
  \mathcal{L}^p(
    \P; 
    C^{ \beta }_b(\overline{D_n}, F )
  )
$
and that for all $x\in D$ it holds $\P$-a.s.\ that $\mathcal{X}(x)=X(x)$.
The fact that 
$\overline{D}=\cup_{n=1}^{\infty}\overline{D_n}$
finally yields for all $\omega\in\Omega$ that
$\mathcal{X}(\omega)\in C(\overline{D},F)$.
Path continuity also implies that $\mathcal{X}$
is $\mathcal{B}(\overline{D})\otimes\mathcal{F}$/$\mathcal{B}(F)$-measurable.
The proof of Proposition~\ref{p:KolChen} is thus completed.
\end{proof}

We emphasize that strong completeness may fail to hold even in the
case of smooth and globally bounded coefficients;
see Li \& Scheutzow~\cite{LiScheutzow2011}.
The following theorem, Theorem~\ref{thm:strong.completeness}
essentially
generalizes
the results in
\cite{CoxHutzenthalerJentzen2013,FangImkellerZhang2007,Li1994,SchenkHoppe1996Deterministic,
ScheutzowSchulze2017,Zhang2010}.

\begin{theorem}[Strong completeness] \label{thm:strong.completeness}
Assume Setting~\ref{s:exists:C0},
assume that $\tau=T$ and  that $V_0,V_1$ are bounded on every bounded subset of $\mathcal{O}$,
let
$\gamma\in(0,\infty)$, $c\in[0,\infty)$
assume that
for all 
$t\in[0,T]$, $ x\in \mathcal{O}$
it holds that
\begin{align} \label{ass:1:moments:X}
\begin{split}
 \max\left\{ \|\mu(t,x)\|_H , \|\sigma(t,x)\|_{\HS(U,H)} \right\}
 \leq  c(1+V_0(x))^{\gamma},
\end{split}
\end{align}
assume that
$\dim(H)<\infty$,
let $ X^x \colon [0,T] \times \Omega \to \mathcal{O} $, $x\in\mathcal{O}$,
be adapted stochastic processes
with continuous sample paths
satisfying that for all $t\in [0,T]$, $x\in\mathcal{O}$
it holds $\P$-a.s.~that 
\begin{align} \label{eq:def:X2}
  X^x_{t}& = 
  x
  + \int_0^{ t } \mu(r,X^x_{r} ) \, dr
  +
  \int_0^{ t } \sigma(r,X^x_{r} ) \, dW_r,
\end{align}
and
assume that $\tfrac{pq}{p+q}\in(\dim(H),\infty)\cap[\tfrac{1}{\gamma},\infty)$.
Then there exists a 
function
$
  \mathcal{X} \colon [0,T] \times \overline{\mathcal{O}} \times \Omega \to \overline{\mathcal{O}}
$
such that
\begin{enumerate}[(i)]
\item $\mathcal{X}$ is
$\mathcal{B}([0,T] \times \overline{\mathcal{O}}) \otimes \mathcal{F} / \mathcal{B}(\overline{\mathcal{O}})$-measurable,
\item it holds for every $ \omega \in \Omega $ that
$
  ([0,T]\times\overline{\mathcal{O}}\ni(t,x)\mapsto
  \mathcal{X}_t^{x}(\omega)\in\overline{\mathcal{O}})
  \in C( [0,T] \times \overline{\mathcal{O}}, \overline{\mathcal{O}}),
$
and
  \item  
for all 
$
  x \in \mathcal{O} 
$ 
it holds $\P$-a.s.~that 
$
  (\mathcal{X}^x_{t })_{t \in [0,T]} 
  = 
  (X^x_{t})_{t \in [0,T]}
$.
\end{enumerate}
\end{theorem}
\begin{proof}[Proof of Theorem~\ref{thm:strong.completeness}]
Throughout this proof let $\delta\in(0,1)$ satisfy
$\tfrac{pq\delta}{p\delta+q}\in(\dim(H),\infty)$ and
let $D_n\subseteq H$, $n\in \N$,
be the sets which satisfy for all $n\in \N$
that $D_n=\{v\in \mathcal{O} \colon \|v\|_H \leq n\}$.
The triangle inequality,
Lemma~\ref{l:temp.reg} (applied 
for every $x\in\mathcal{O}$, $n\in\N$
with 
$s=0$
in the notation of
Lemma~\ref{l:temp.reg}),
and boundedness of $V_0$ on the bounded subsets $D_n$, $n\in\N$, of $\mathcal{O}$
show for all $n\in \N$ with $D_n\neq \emptyset$ that
\begin{align} \label{thm:strong.completeness:X}
\begin{split}
&
\sup_{x\in D_n} 
\Big\|\sup_{t\in[0,T]} \| X_{t}^{x} \|_H \Big\|_{L^{\frac{pq\delta }{p\delta +q}}(\P;\R)}
\leq
\sup_{x\in D_n} 
\Big\| \sup_{t\in[0,T]} \| X_{t}^{x}-x \|_H \Big\|_{L^{\frac{pq}{p+q}}(\P;\R)}
+
\sup_{x\in D_n} \|x\|_H
\\&\leq
    \sqrt{T}c e^{\alpha_0\gamma T}
    \Big|\tfrac{pq}{p+q}\gamma+\sup_{y\in D_n}V_0(y)\Big|^{\gamma}
    \big(\sqrt{T}+\tfrac{pq}{p+q}\big)\sqrt{T}
    +n<\infty.
\end{split}
\end{align} 
Moreover, item (ii) in Lemma~\ref{l:localLipschitz.initial}
and boundedness of $V_0$ on the bounded subsets $D_n$, $n\in\N$, of $\mathcal{O}$
imply for all $n\in \N$ with $\# D_n\in[2,\infty]$ that
\begin{equation}  \begin{split}
&
  \sup_{ \substack{ x_1,x_2 \in D_n\\
  x_1\neq x_2}}
  \frac{ 
  \left\|\sup_{t\in[0,T]}\big\| X_{t}^{x_1} - X_{t}^{x_2}\big\|_H\right\|_{L^{\frac{pq\delta}{p\delta+q}}(\P;\R)}
  }
  {\|x_1-x_2\|_H}
  \leq
     \left(\tfrac{1}{\delta}\smallint_{\frac{1-\delta}{\delta}}^{\infty}\tfrac{s^{\delta}}{(s+1)^2}\,ds+1\right)^{\!\frac{1}{2\delta }}
\\&\qquad
  \cdot\exp\!\left(
        \int_0^{T}  
        \constFun(r)
        +
        \tfrac{
          \beta_{ 0 } 
        }{
          q_{0} e^{ \alpha_{ 0 } r }
        }
        + 
        \tfrac{ \beta_{ 1 } }{ 	q_{ 1 } e^{ \alpha_{ 1 } r } }
        \,
        dr
        +
        \sum_{ i = 0 }^1 
        \tfrac{\sup_{y\in D_n}V_i(y)+\sup_{y\in D_n}V_i(y)}{2 q_{i} } 
      \right)
  <\infty.
\end{split}     \end{equation}
Proposition~\ref{p:KolChen}
(applied with 
$H = H$,
$D=\mathcal{O}$,
$E=C([0,T],H)$,
$F=C([0,T],\overline{\mathcal{O}})$,
$p=\tfrac{pq\delta}{p\delta+q}$,
$\alpha=1$
in the notation of Proposition~\ref{p:KolChen})
finally yields the assertion.
This
completes the proof of Theorem~\ref{thm:strong.completeness}.
\end{proof}

\subsection{Perturbation estimates for SDEs}\label{sec:perturbation}
Many problems can be formulated as perturbations of SDEs, e.g.:
time discretizations of SDEs, spatial discretizations of stochastic partial differential equations (SPDEs), or
small noise approximations of ordinary differential equations.
We follow here the principal perturbation
approach of Hutzenthaler \& Jentzen~\cite{HutzenthalerJentzen2014}.
The following corollary, Corollary~\ref{c:perturbation}, 
applies Theorem~\ref{thm:moments:hilbert.m}
to derive a suitable perturbation estimate.
The marginal perturbation estimate~\eqref{eq:perturbation.marginal} below
is a minor improvement of Hutzenthaler \& Jentzen~\cite[Theorem 1.2]{HutzenthalerJentzen2014}.
To the best of our knowledge, the uniform perturbation
estimate~\eqref{eq:perturbation.uniform} is new.

\begin{cor}[Perturbation estimate for SDEs]\label{c:perturbation}
  Assume Setting~\ref{sett:applications},
  let $a\colon[0,T]\times\Omega\to H$ and $b\colon[0,T]\times\Omega\to\HS(U,H)$
  be $\mathcal{B}([0,T])\otimes\mathcal{F}$-measurable and adapted stochastic processes
  which satisfy $\P$-a.s.\ that $\int_0^{\tau}\|a_s\|_H+\|b_s\|^2_{\HS(U,H)}\,ds<\infty$
  and 
  which satisfy that for all $t\in[0,T]$ it holds $\P$-a.s.\ that
  \begin{equation}  \begin{split}\label{eq:appl.Y}
    Y_{\min\{t,\tau\}}=Y_0+\int_0^t\1_{[0,\tau]}(s)a_s\,ds+\int_0^t\1_{[0,\tau]}(s)b_s\,dW_s,
  \end{split}     \end{equation}
  let $p\in[2,\infty)$, $\eps\in[0,\infty]$, $\delta\in(0,\infty]$,
  and
  assume that $\P$-a.s.\ it holds that
   $\int_0^{\tau}\max\{\langle X_s-Y_s,\mu(s,X_s)-\mu(s,Y_s)\rangle_H+\tfrac{1+\eps}{2}\|\sigma(s,X_s)-\sigma(s,Y_s)\|_{\HS(U,H)}^2
   +\tfrac{(p-2)(1+\eps)}{2}\tfrac{\|\langle X_s-Y_s,\sigma(s,X_s)-\sigma(s,Y_s)\rangle_H\|_{\HS(U,\R)}^2}{\|X_s-Y_s\|_H^2},0\}/\|X_s-Y_s\|_H^2\,ds<\infty
   $.
  Then
  \begin{enumerate}[(i)]
    \item  
     it holds for all $q_1,q_2\in(0,\infty]$ with $\tfrac{1}{q_1}=\tfrac{1}{q_2}+\tfrac{1}{p}$
    that
    \begin{equation}  \begin{split}\label{eq:perturbation.marginal}
      &\|X_{\tau}-Y_{\tau}\|_{L^{q_1}(\P;H)}
      \\&
      \leq 
      \bigg(\|X_0-Y_0\|_{L^p(\P;H)}^2
      \\&\qquad
      +2\int_0^{T} 
      \left\|\1_{[0,\tau]}(t)\big(\delta\|\mu(t,Y_t)-a_t\|_H^2
      +\tfrac{p-1}{2}(1+\tfrac{1}{\eps})\|\sigma(t,Y_t)-b_t\|_{\HS(U,H)}^2\big)
      \right\|_{L^{\frac{p}{2}}(\P;\R)}\,dt
      \bigg)^{\frac{1}{2}}
      \\&\quad
      \cdot
      \bigg\|\exp\bigg(\int_0^{\tau}
  \max\Big\{
    \tfrac{\langle X_t-Y_t,\mu(t,X_t)-\mu(t,Y_t) \rangle_{H}
  +\frac{1+\eps}{2}\left\|\sigma(t,X_t)-\sigma(t,Y_t)\right\|_{\HS(U,H)}^2
  }{\|X_t-Y_t\|_{H}^2}
  \\&\qquad\qquad\qquad\qquad\qquad
  +\tfrac{(p-2)(1+\eps)\|\langle X_t-Y_t,\sigma(t,X_t)-\sigma(t,Y_t)\rangle_{H}\|_{\HS(U,\R)}^2
  }{2\|X_t-Y_t\|_{H}^4}+\tfrac{1}{4\delta},
  0\Big\} \,dt
  \bigg)\bigg\|_{L^{q_2}(\P;\R)}
\end{split}     \end{equation}
and
\item
 it holds for all $q_1,q_2,q_3\in(0,\infty]$ with $q_3<p$ and $\tfrac{1}{q_1}=\tfrac{1}{q_2}+\tfrac{1}{q_3}$
    that
    \begin{equation}  \begin{split}\label{eq:perturbation.uniform}
      &\left\|\sup_{s\in[0,\tau]}\|X_{s}-Y_{s}\|_H\right\|_{L^{q_1}(\P;\R)}
      \\&
      \leq \left\|\left(\|X_0-Y_0\|_H^2+2\int_0^{\tau} \left|\delta\|\mu(t,Y_t)-a_t\|_H^2
      +\tfrac{p-1}{2}(1+\tfrac{1}{\eps})\|\sigma(t,Y_t)-b_t\|_{\HS(U,H)}^2\right|\,dt
      \right)^{\!\frac{1}{2}}
      \right\|_{L^{q_3}(\P;\R)}
      \\&\quad
      \cdot
     \bigg(\tfrac{p}{q_3}\smallint_{\frac{p-q_3}{q_3}}^{\infty}\tfrac{s^{\frac{q_3}{p}}}{(s+1)^2}\,ds+1\bigg)^{\frac{p}{2q_3}}
      \bigg\|\exp\bigg(\int_0^{\tau}
  \max\Big\{
    \tfrac{\langle X_t-Y_t,\mu(t,X_t)-\mu(t,Y_t) \rangle_{H}
  +\frac{1+\eps}{2}\left\|\sigma(t,X_t)-\sigma(t,Y_t)\right\|_{\HS(U,H)}^2
  }{\|X_t-Y_t\|_{H}^2}
  \\&\qquad\qquad\qquad\qquad\qquad\qquad\quad
  +\tfrac{(p-2)(1+\eps)\|\langle X_t-Y_t,\sigma(t,X_t)-\sigma(t,Y_t)\rangle_{H}\|_{\HS(U,\R)}^2
  }{2\|X_t-Y_t\|_{H}^4}+\tfrac{1}{4\delta},
  0\Big\} \,dt
  \bigg)\bigg\|_{L^{q_2}(\P;\R)}.
    \end{split}     \end{equation}
  \end{enumerate}

\end{cor}
\begin{proof}[Proof of Corollary~\ref{c:perturbation}]
  Without loss of generality we assume that
  \begin{equation}  \begin{split}\label{eq:assumeD}
    \P\Big(\int_0^{\tau}\eps\|\sigma(t,X_t)-\sigma(t,Y_t)\|_{\HS(U,H)}^2
    +\tfrac{1}{\eps}\|\sigma(t,Y_t)-b_t\|_{\HS(U,H)}^2
    +\delta\|\mu(t,Y_t)-a_t\|_H^2\,dt<\infty\Big)=1
  \end{split}     \end{equation}
  (otherwise the assertion is trivial).
  First, \eqref{eq:appl.X} and~\eqref{eq:appl.Y} imply that for all $t\in[0,T]$
  it holds $\P$-a.s.\ that
  \begin{equation}  \begin{split}\label{eq:perturbation.XY}
    X_{\min\{t,\tau\}}-Y_{\min\{t,\tau\}}
    =\int_0^t\1_{[0,\tau]}(s)(\mu(s,X_s)-a_s)\,ds
    +\int_0^t\1_{[0,\tau]}(s)(\sigma(s,X_s)-b_s)\,dW_s.
  \end{split}     \end{equation}
  Moreover, it holds for all $t\in[0,T]$ that
  \begin{equation}  \begin{split}
  &
  \tfrac{\langle X_t-Y_t,\mu(t,X_t)-a_t \rangle_{H}
  +\frac{1}{2}\|\sigma(t,X_t)-b_t\|^2_{\HS(U,H)}
  +\frac{p-2}{2}\frac{\|\langle X_t-Y_t,\sigma(t,X_t)-b_t\rangle_{H}\|_{\HS(U,\R)}^2
  }{\|X_t-Y_t\|_{H}^2}
  }{\|X_t-Y_t\|_{H}^2}
  \\&
  =
  \tfrac{\langle X_t-Y_t,\mu(t,X_t)-\mu(t,Y_t) \rangle_{H}
  +\tfrac{1}{2}\left\|\sigma(t,X_t)-\sigma(t,Y_t)\right\|_{\HS(U,H)}^2
  +\frac{p-2}{2}\frac{\|\langle X_t-Y_t,\sigma(t,X_t)-\sigma(t,Y_t)\rangle_{H}\|_{\HS(U,\R)}^2
  }{\|X_t-Y_t\|_{H}^2}
  }{\|X_t-Y_t\|_{H}^2}
  \\&\quad
  +
  \tfrac{\langle X_t-Y_t,\mu(t,Y_t)-a_t \rangle_{H}
  +\tfrac{1}{2}\left\|\sigma(t,Y_t)-b_t\right\|_{\HS(U,H)}^2
  +\langle \sigma(t,X_t)-\sigma(t,Y_t),\sigma(t,Y_t)-b_t\rangle_{\HS(U,H)}
  }{\|X_t-Y_t\|_{H}^2}
  \\&\quad
  +
  \tfrac{p-2}{2}\tfrac{
  \|\langle X_t-Y_t,\sigma(t,Y_t)-b_t\rangle_{H}\|_{\HS(U,\R)}^2
  +2\left\langle
      \langle X_t-Y_t,\sigma(t,X_t)-\sigma(t,Y_t)\rangle_{H},
      \langle X_t-Y_t,\sigma(t,Y_t)-b_t\rangle_{H}
  \right\rangle_{\HS(U,\R)}
  }{\|X_t-Y_t\|_{H}^4}.
  \end{split}     \end{equation}
  This, the Cauchy-Schwarz inequality, and Young's inequality
  yield for all $t\in[0,T]$ that
  \begin{equation}  \begin{split}
  &
  \tfrac{\langle X_t-Y_t,\mu(t,X_t)-a_t \rangle_{H}
  +\frac{1}{2}\|\sigma(t,X_t)-b_t\|^2_{\HS(U,H)}
  +\frac{p-2}{2}\frac{\|\langle X_t-Y_t,\sigma(t,X_t)-b_t\rangle_{H}\|_{\HS(U,\R)}^2
  }{\|X_t-Y_t\|_{H}^2}
  }{\|X_t-Y_t\|_{H}^2}
  \\&
  \leq
  \tfrac{\langle X_t-Y_t,\mu(t,X_t)-\mu(t,Y_t) \rangle_{H}
  +\frac{1+\eps}{2}\left\|\sigma(t,X_t)-\sigma(t,Y_t)\right\|_{\HS(U,H)}^2
  +\frac{p-2}{2}(1+\eps)\frac{\|\langle X_t-Y_t,\sigma(t,X_t)-\sigma(t,Y_t)\rangle_{H}\|_{\HS(U,\R)}^2
  }{\|X_t-Y_t\|_{H}^2}
  }{\|X_t-Y_t\|_{H}^2}+\tfrac{1}{4\delta}
  \\&\quad
  +
  \tfrac{\delta\|\mu(t,Y_t)-a_t \|_{H}^2
  +\frac{1}{2}(1+\frac{1}{\eps})\left\|\sigma(t,Y_t)-b_t\right\|_{\HS(U,H)}^2
  }{\|X_t-Y_t\|_{H}^2}
  +
  \tfrac{
  (p-2)(1+\frac{1}{\eps})\|\langle X_t-Y_t,\sigma(t,Y_t)-b_t\rangle_{H}\|_{\HS(U,\R)}^2
  }{2\|X_t-Y_t\|_{H}^4}
  \\&
  \leq
  \tfrac{\langle X_t-Y_t,\mu(t,X_t)-\mu(t,Y_t) \rangle_{H}
  +\frac{1+\eps}{2}\left\|\sigma(t,X_t)-\sigma(t,Y_t)\right\|_{\HS(U,H)}^2
  +\frac{p-2}{2}(1+\eps)\frac{\|\langle X_t-Y_t,\sigma(t,X_t)-\sigma(t,Y_t)\rangle_{H}\|_{\HS(U,\R)}^2
  }{\|X_t-Y_t\|_{H}^2}
  }{\|X_t-Y_t\|_{H}^2}+\tfrac{1}{4\delta}
  \\&\quad
  +
  \tfrac{\delta\|\mu(t,Y_t)-a_t \|_{H}^2
  +\frac{p-1}{2}(1+\frac{1}{\eps})\|\sigma(t,Y_t)-b_t\|_{\HS(U,H)}^2
  }{\|X_t-Y_t\|_{H}^2}.
  \end{split}     \end{equation}
  This, \eqref{eq:perturbation.XY}, \eqref{eq:assumeD},
  and Corollary~\ref{c:moments:hilbert.m}
  (applied
  with $X_t=X_{t}-Y_{t}$, $a_t=\mu(t,X_t)-a_t$, $b_t=\sigma(t,X_t)-b_t$,
  \begin{equation}  \begin{split}
  \alpha_t
 & =
  \max\Big\{\tfrac{\langle X_t-Y_t,\mu(t,X_t)-\mu(t,Y_t) \rangle_{H}
  +\frac{1+\eps}{2}\left\|\sigma(t,X_t)-\sigma(t,Y_t)\right\|_{\HS(U,H)}^2
  +\frac{p-2}{2}(1+\eps)\frac{\|\langle X_t-Y_t,\sigma(t,X_t)-\sigma(t,Y_t)\rangle_{H}\|_{\HS(U,\R)}^2
  }{\|X_t-Y_t\|_{H}^2}
  }{\|X_t-Y_t\|_{H}^2}
  \\&\qquad\qquad\qquad\qquad
  +\tfrac{1}{4\delta},0\Big\},
  \end{split}     \end{equation}
  $\beta_t=
  \sqrt{2}\big(\delta\|\mu(t,Y_t)-a_t \|_{H}^2
  +\frac{p-1}{2}(1+\frac{1}{\eps})\|\sigma(t,Y_t)-b_t\|_{\HS(U,H)}^2\big)^{\frac{1}{2}}
  $
  for all $t\in[0,T]$
  in the notation of
  Corollary~\ref{c:moments:hilbert.m})
  imply that
     it holds for all $q_1,q_2\in(0,\infty]$, $\delta\in(0,\infty)$ with 
     $\tfrac{1}{q_1}=\tfrac{1}{q_2}+\tfrac{1}{p}$
    that
    \begin{equation}  \begin{split}
      &\|X_{\tau}-Y_{\tau}\|_{L^{q_1}(\P;H)}
      \\&
      \leq 
      \bigg(\|X_0-Y_0\|_{L^p(\P;H)}^2
      \\&\qquad
      +2\int_0^{T} 
      \left\|\1_{[0,\tau]}(t)\big(\delta\|\mu(t,Y_t)-a_t\|_H^2
      +\tfrac{p-1}{2}\tfrac{1+\eps}{\eps}\|\sigma(t,Y_t)-b_t\|_{\HS(U,H)}^2\big)
      \right\|_{L^p(\P;\R)}^2\,dt
      \bigg)^{\frac{1}{2}}
      \\&\quad
      \cdot
      \bigg\|\exp\bigg(\int_0^{\tau}
  \max\Big\{
    \tfrac{\langle X_t-Y_t,\mu(t,X_t)-\mu(t,Y_t) \rangle_{H}
  +\frac{1+\eps}{2}\left\|\sigma(t,X_t)-\sigma(t,Y_t)\right\|_{\HS(U,H)}^2
  }{\|X_t-Y_t\|_{H}^2}
  \\&\qquad\qquad\qquad\qquad\qquad
  +\tfrac{(p-2)(1+\eps)\|\langle X_t-Y_t,\sigma(t,X_t)-\sigma(t,Y_t)\rangle_{H}\|_{\HS(U,\R)}^2
  }{2\|X_t-Y_t\|_{H}^4}+\tfrac{1}{4\delta},
  0\Big\} \,dt
  \bigg)\bigg\|_{L^{q_2}(\P;\R)}
\end{split}     \end{equation}
and it holds for all $q_1,q_2,q_3\in(0,\infty]$ with $q_3<p$ and $\tfrac{1}{q_1}=\tfrac{1}{q_2}+\tfrac{1}{q_3}$
    that
    \begin{equation}  \begin{split}
      &\left\|\sup_{s\in[0,\tau]}\|X_{s}-Y_{s}\|_H\right\|_{L^{q_1}(\P;\R)}
      \\&
      \leq \left\|\left(\|X_0-Y_0\|_H^2+2\int_0^{\tau} \delta\|\mu(t,Y_t)-a_t\|_H^2
      +\tfrac{p-1}{2}(1+\tfrac{1}{\eps})\|\sigma(t,Y_t)-b_t\|_{\HS(U,H)}^2\,dt
      \right)^{\!\frac{1}{2}}
      \right\|_{L^{q_3}(\P;\R)}
      \\&\quad
      \cdot
     \bigg(\tfrac{p}{q_3}\smallint_{\frac{p-q_3}{q_3}}^{\infty}\tfrac{s^{\frac{q_3}{p}}}{(s+1)^2}\,ds+1\bigg)^{\frac{p}{2q_3}}
      \bigg\|\exp\bigg(\int_0^{\tau}
  \max\Big\{
    \tfrac{\langle X_t-Y_t,\mu(t,X_t)-\mu(t,Y_t) \rangle_{H}
  +\frac{1+\eps}{2}\left\|\sigma(t,X_t)-\sigma(t,Y_t)\right\|_{\HS(U,H)}^2
  }{\|X_t-Y_t\|_{H}^2}
  \\&\qquad\qquad\qquad\qquad\qquad\qquad\quad
  +\tfrac{(p-2)(1+\eps)\|\langle X_t-Y_t,\sigma(t,X_t)-\sigma(t,Y_t)\rangle_{H}\|_{\HS(U,\R)}^2
  }{2\|X_t-Y_t\|_{H}^4}+\tfrac{1}{4\delta},
  0\Big\} \,dt
  \bigg)\bigg\|_{L^{q_2}(\P;\R)}.
    \end{split}     \end{equation}
  This completes the proof of Corollary~\ref{c:perturbation}.
\end{proof}
For example,
Corollary~\ref{c:perturbation} can be applied to prove
strong convergence rates for implementable approximations
of solutions of SDEs.
The classical (exponential) Euler approximations diverge
in the strong
and weak sense for most one-dimensional SDEs with
super-linearly growing coefficients
(see \cite{hjk11,HutzenthalerJentzenKloeden2013})
and also for some SPDEs (see Beccari et al.\cite{BHJKLS2019}).
It was shown in \cite{hjk12,HutzenthalerJentzen2015}
that minor modifications of the Euler method -- so called
tamed Euler methods -- avoid this divergence problem;
see also the Euler-type methods, e.g., in
\cite{BeynIsaakKruse2016,
BloemkerSchillingsWacker2018,
BrosseDurmusMoulinesSabanis2018,
ChassagneuxJacquierMihaylov2016,
ChenGanWang2018,
DiazJerez2017linear,
GuoLiuMao2018,
HanMaDing2019,
HuLiMao2018,
KumarKumar2018,
LanXia2018,
LiuMao2013,
MoraEtAl2017,
NgoLuong2017,
SabanisZhang2019,
WangGan2013,
ZhanJiangLiu2018,
ZhouEtAl2016,
ZongWuHuang2014}.
Now, analogously to Hutzenthaler \& Jentzen~\cite{HutzenthalerJentzen2014},
Corollary~\ref{c:perturbation} is a powerful tool
to establish uniform strong convergence rates
(in combination with
exponential moment estimates for suitably
tamed Euler approximations, e.g., Hutzenthaler et al.~\cite{HutzenthalerJentzenWang2018}).

\subsubsection*{Acknowledgement}
This project has been partially supported 
by the Deutsche Forschungsgesellschaft (DFG) via 
RTG 2131 {\it High-dimensional Phenomena in Probability -- Fluctuations and Discontinuity}.


\begin{thebibliography}{48}
\providecommand{\natexlab}[1]{#1}
\providecommand{\url}[1]{\texttt{#1}}
\expandafter\ifx\csname urlstyle\endcsname\relax
  \providecommand{\doi}[1]{doi: #1}\else
  \providecommand{\doi}{doi: \begingroup \urlstyle{rm}\Url}\fi

\bibitem[Ba{\~n}uelos and Os\c{e}kowski(2014)]{BanuelosOsekowski2014}
Rodrigo Ba{\~n}uelos and Adam Os\c{e}kowski.
\newblock Sharp maximal {$L^p$}-estimates for martingales.
\newblock \emph{Illinois Journal of Mathematics}, 58\penalty0 (1):\penalty0
  149--165, 2014.

\bibitem[Beccari et~al.(2019)Beccari, Hutzenthaler, Jentzen, Kurniawan,
  Lindner, and Salimova]{BHJKLS2019}
Matteo Beccari, Martin Hutzenthaler, Arnulf Jentzen, Ryan Kurniawan, Felix
  Lindner, and Diyora Salimova.
\newblock Strong and weak divergence of exponential and linear-implicit {E}uler
  approximations for stochastic partial differential equations with
  superlinearly growing nonlinearities.
\newblock \emph{arXiv preprint arXiv:1903.06066}, 2019.

\bibitem[Beesack(1977)]{Beesack1977}
Paul~R. Beesack.
\newblock {On Lakshmikantham's comparison method for Gronwall inequalities}.
\newblock In \emph{Annales Polonici Mathematici}, volume~2, pages 187--222,
  1977.

\bibitem[Beyn et~al.(2016)Beyn, Isaak, and Kruse]{BeynIsaakKruse2016}
Wolf-J{\"u}rgen Beyn, Elena Isaak, and Raphael Kruse.
\newblock {Stochastic C-stability and B-consistency of explicit and implicit
  Euler-type schemes}.
\newblock \emph{Journal of Scientific Computing}, 67\penalty0 (3):\penalty0
  955--987, 2016.

\bibitem[Bl\"omker et~al.(2018)Bl\"omker, Schillings, and
  Wacker]{BloemkerSchillingsWacker2018}
Dirk Bl\"omker, Claudia Schillings, and Philipp Wacker.
\newblock A strongly convergent numerical scheme from ensemble {K}alman
  inversion.
\newblock \emph{SIAM Journal on Numerical Analysis}, 56\penalty0 (4):\penalty0
  2537--2562, 2018.

\bibitem[Brosse et~al.(2018)Brosse, Durmus, Moulines, and
  Sabanis]{BrosseDurmusMoulinesSabanis2018}
Nicolas Brosse, Alain Durmus, {\'E}ric Moulines, and Sotirios Sabanis.
\newblock {The tamed unadjusted Langevin algorithm}.
\newblock \emph{Stochastic Processes and their Applications}, 2018.

\bibitem[Burkholder(1975)]{Burkholder1975}
Donald~L. Burkholder.
\newblock One-sided maximal functions and {$H^p$}.
\newblock \emph{Journal of functional analysis}, 18\penalty0 (4):\penalty0
  429--454, 1975.

\bibitem[Chassagneux et~al.(2016)Chassagneux, Jacquier, and
  Mihaylov]{ChassagneuxJacquierMihaylov2016}
Jean-Fran{\c{c}}ois Chassagneux, Antoine Jacquier, and Ivo Mihaylov.
\newblock {An explicit Euler scheme with strong rate of convergence for
  financial SDEs with non-Lipschitz coefficients}.
\newblock \emph{SIAM Journal on Financial Mathematics}, 7\penalty0
  (1):\penalty0 993--1021, 2016.

\bibitem[Chen et~al.(2018)Chen, Gan, and Wang]{ChenGanWang2018}
Ziheng Chen, Siqing Gan, and Xiaojie Wang.
\newblock {Mean-square approximations of L\`evy noise driven SDEs with
  super-linearly growing diffusion and jump coefficients}.
\newblock \emph{arXiv preprint arXiv:1812.03069}, 2018.

\bibitem[Cox et~al.(2013)Cox, Hutzenthaler, and
  Jentzen]{CoxHutzenthalerJentzen2013}
Sonja~G. Cox, Martin Hutzenthaler, and Arnulf Jentzen.
\newblock Local {L}ipschitz continuity in the initial value and strong
  completeness for nonlinear stochastic differential equations.
\newblock \emph{Mem. Amer. Math. Soc. \textup{(accepted subject to revision,
  arXiv:1309.5595)}}, pages 1--54, 2013.

\bibitem[{Da Prato} and Zabczyk(1992)]{dz92}
Giuseppe {Da Prato} and Jerzy Zabczyk.
\newblock \emph{Stochastic equations in infinite dimensions}, volume~44 of
  \emph{Encyclopedia of Mathematics and its Applications}.
\newblock Cambridge University Press, Cambridge, 1992.

\bibitem[D{\'\i}az-Infante and Jerez(2017)]{DiazJerez2017linear}
Sa{\'u}l D{\'\i}az-Infante and S~Jerez.
\newblock {The linear Steklov method for SDEs with non-globally Lipschitz
  coefficients: Strong convergence and simulation}.
\newblock \emph{Journal of Computational and Applied Mathematics},
  309:\penalty0 408--423, 2017.

\bibitem[Fang et~al.(2007)Fang, Imkeller, and Zhang]{FangImkellerZhang2007}
Shizan Fang, Peter Imkeller, and Tusheng Zhang.
\newblock Global flows for stochastic differential equations without global
  {L}ipschitz conditions.
\newblock \emph{Ann. Probab.}, 35\penalty0 (1):\penalty0 180--205, 2007.

\bibitem[Grasselli(2017)]{Grasselli2017}
Martino Grasselli.
\newblock {The 4/2 stochastic volatility model: a unified approach for the
  Heston and the 3/2 model}.
\newblock \emph{Mathematical Finance}, 27\penalty0 (4):\penalty0 1013--1034,
  2017.

\bibitem[Grohs et~al.(2018)Grohs, Hornung, Jentzen, and
  Von~Wurstemberger]{GrohsHornungJentzenWurstemberger2018}
Philipp Grohs, Fabian Hornung, Arnulf Jentzen, and Philippe Von~Wurstemberger.
\newblock {A proof that artificial neural networks overcome the curse of
  dimensionality in the numerical approximation of Black-Scholes partial
  differential equations}.
\newblock \emph{arXiv preprint arXiv:1809.02362}, 2018.

\bibitem[Guo et~al.(2018)Guo, Liu, and Mao]{GuoLiuMao2018}
Qian Guo, Wei Liu, and Xuerong Mao.
\newblock {A note on the partially truncated Euler--Maruyama method}.
\newblock \emph{Applied Numerical Mathematics}, 130:\penalty0 157--170, 2018.

\bibitem[Gy{\"o}ngy and Krylov(1996)]{GyoengyKrylov1996}
Istv{\'a}n Gy{\"o}ngy and Nicolai Krylov.
\newblock Existence of strong solutions for {I}t\^o's stochastic equations via
  approximations.
\newblock \emph{Probab. Theory Related Fields}, 105\penalty0 (2):\penalty0
  143--158, 1996.

\bibitem[Hairer et~al.(2015)Hairer, Hutzenthaler, and
  Jentzen]{HairerHutzenthalerJentzen2015}
Martin Hairer, Martin Hutzenthaler, and Arnulf Jentzen.
\newblock Loss of regularity for {K}olmogorov equations.
\newblock \emph{Ann. Probab.}, 43\penalty0 (2):\penalty0 468--527, 2015.

\bibitem[Han et~al.(2019)Han, Ma, and Ding]{HanMaDing2019}
Minggang Han, Qiang Ma, and Xiaohua Ding.
\newblock {The projected explicit It{\^o}--Taylor methods for stochastic
  differential equations under locally Lipschitz conditions and polynomial
  growth conditions}.
\newblock \emph{Journal of Computational and Applied Mathematics},
  348:\penalty0 161--180, 2019.

\bibitem[Heston(1997)]{Heston1997}
Steven~L. Heston.
\newblock A simple new formula for options with stochastic volatility.
\newblock Technical report, Washington University of St Louis working paper,
  1997.

\bibitem[Hu et~al.(2018)Hu, Li, and Mao]{HuLiMao2018}
Liangjian Hu, Xiaoyue Li, and Xuerong Mao.
\newblock {Convergence rate and stability of the truncated Euler--Maruyama
  method for stochastic differential equations}.
\newblock \emph{Journal of Computational and Applied Mathematics},
  337:\penalty0 274--289, 2018.

\bibitem[Hudde et~al.(2019)Hudde, Hutzenthaler, and
  Mazzonetto]{HuddeHutzenthalerMazzonetto2019b}
Anselm Hudde, Martin Hutzenthaler, and Sara Mazzonetto.
\newblock Existence of spatially differentiable solutions of stochastic
  differential equations with non-globally monotone coefficient functions.
\newblock \emph{arXiv preprint arXiv:1903.09707}, 2019.

\bibitem[Hutzenthaler and Jentzen(2014)]{HutzenthalerJentzen2014}
Martin Hutzenthaler and Arnulf Jentzen.
\newblock On a perturbation theory and on strong convergence rates for
  stochastic ordinary and partial differential equations with non-globally
  monotone coefficients.
\newblock \emph{Annals of Probability \textup{(accepted subject to revision,
  arXiv:1401.0295)}}, pages 1--41, 2014.

\bibitem[Hutzenthaler and Jentzen(2015)]{HutzenthalerJentzen2015}
Martin Hutzenthaler and Arnulf Jentzen.
\newblock Numerical approximations of stochastic differential equations with
  non-globally {L}ipschitz continuous coefficients.
\newblock \emph{Mem. Amer. Math. Soc.}, 4:\penalty0 1--112, 2015.

\bibitem[Hutzenthaler et~al.(2011)Hutzenthaler, Jentzen, and Kloeden]{hjk11}
Martin Hutzenthaler, Arnulf Jentzen, and Peter~E. Kloeden.
\newblock Strong and weak divergence in finite time of {E}uler's method for
  stochastic differential equations with non-globally {L}ipschitz continuous
  coefficients.
\newblock \emph{Proc. R. Soc. Lond. Ser. A Math. Phys. Eng. Sci.},
  467:\penalty0 1563--1576, 2011.

\bibitem[Hutzenthaler et~al.(2012)Hutzenthaler, Jentzen, and Kloeden]{hjk12}
Martin Hutzenthaler, Arnulf Jentzen, and Peter~E. Kloeden.
\newblock Strong convergence of an explicit numerical method for {SDE}s with
  nonglobally {L}ipschitz continuous coefficients.
\newblock \emph{Ann. Appl. Probab.}, 22\penalty0 (4):\penalty0 1611--1641,
  2012.

\bibitem[Hutzenthaler et~al.(2013)Hutzenthaler, Jentzen, and
  Kloeden]{HutzenthalerJentzenKloeden2013}
Martin Hutzenthaler, Arnulf Jentzen, and Peter~E. Kloeden.
\newblock Divergence of the multilevel {M}onte {C}arlo {E}uler method for
  nonlinear stochastic differential equations.
\newblock \emph{Ann. Appl. Probab.}, 23\penalty0 (5):\penalty0 1913--1966,
  2013.

\bibitem[Hutzenthaler et~al.(2018)Hutzenthaler, Jentzen, and
  Wang]{HutzenthalerJentzenWang2018}
Martin Hutzenthaler, Arnulf Jentzen, and Xiaojie Wang.
\newblock Exponential integrability properties of numerical approximation
  processes for nonlinear stochastic differential equations.
\newblock \emph{Math. Comp.}, 87\penalty0 (311):\penalty0 1353--1413, 2018.

\bibitem[Kumar and Kumar(2018)]{KumarKumar2018}
Tejinder Kumar and Chaman Kumar.
\newblock {A New Efficient Explicit Scheme of Order $1.5 $ for SDE with
  Super-linear Drift Coefficient}.
\newblock \emph{arXiv preprint arXiv:1805.07976}, 2018.

\bibitem[Lan and Xia(2018)]{LanXia2018}
Guangqiang Lan and Fang Xia.
\newblock {Strong convergence rates of modified truncated EM method for
  stochastic differential equations}.
\newblock \emph{Journal of Computational and Applied Mathematics},
  334:\penalty0 1--17, 2018.

\bibitem[Li(1994)]{Li1994}
Xue-Mei Li.
\newblock Strong {$p$}-completeness of stochastic differential equations and
  the existence of smooth flows on noncompact manifolds.
\newblock \emph{Probab. Theory Related Fields}, 100\penalty0 (4):\penalty0
  485--511, 1994.

\bibitem[Li and Scheutzow(2011)]{LiScheutzow2011}
Xue-Mei Li and Michael Scheutzow.
\newblock Lack of strong completeness for stochastic flows.
\newblock \emph{Ann. Probab.}, 39\penalty0 (4):\penalty0 1407--1421, 2011.

\bibitem[Liu and Mao(2013)]{LiuMao2013}
Wei Liu and Xuerong Mao.
\newblock {Strong convergence of the stopped Euler--Maruyama method for
  nonlinear stochastic differential equations}.
\newblock \emph{Applied Mathematics and Computation}, 223:\penalty0 389--400,
  2013.

\bibitem[Mittmann and Steinwart(2003)]{MS03}
Katrin Mittmann and Ingo Steinwart.
\newblock On the existence of continuous modifications of vector-valued random
  fields.
\newblock \emph{Georgian Mathematical Journal}, 10\penalty0 (2):\penalty0
  311--317, 2003.

\bibitem[Mora et~al.(2017)Mora, Mardones, Jim{\'e}nez, Selva, and
  Biscay]{MoraEtAl2017}
Carlos~M Mora, HA~Mardones, Juan~C Jim{\'e}nez, M~Selva, and R~Biscay.
\newblock {A stable numerical scheme for stochastic differential equations with
  multiplicative noise}.
\newblock \emph{SIAM Journal on Numerical Analysis}, 55\penalty0 (4):\penalty0
  1614--1649, 2017.

\bibitem[Ngo and Luong(2017)]{NgoLuong2017}
Hoang-Long Ngo and Duc-Trong Luong.
\newblock {Strong rate of tamed Euler-Maruyama approximation for stochastic
  differential equations with H{\"o}lder continuous diffusion coefficient}.
\newblock \emph{Braz. J. Probab. Stat}, 31\penalty0 (1):\penalty0 24--40, 2017.

\bibitem[Opial(1957)]{Opial1957}
Z~Opial.
\newblock {Sur un syst{\`e}me d'in{\'e}galit{\'e}s int{\'e}grales}.
\newblock In \emph{Annales Polonici Mathematici}, volume~2, pages 200--209,
  1957.

\bibitem[Platen(1997)]{Platen1997}
Eckhard Platen.
\newblock A non-linear stochastic volatility model.
\newblock \emph{Research report: financial mathematics research report/Centre
  for mathematics and its applications (Canberra)}, 1997.

\bibitem[Sabanis and Zhang(2019)]{SabanisZhang2019}
Sotirios Sabanis and Ying Zhang.
\newblock {On explicit order 1.5 approximations with varying coefficients: the
  case of super-linear diffusion coefficients}.
\newblock \emph{Journal of Complexity}, 50:\penalty0 84--115, 2019.

\bibitem[Schenk-Hopp{\'e}(1996)]{SchenkHoppe1996Deterministic}
Klaus~R. Schenk-Hopp{\'e}.
\newblock Deterministic and stochastic {D}uffing-van der {P}ol oscillators are
  non-explosive.
\newblock \emph{Z. Angew. Math. Phys.}, 47\penalty0 (5):\penalty0 740--759,
  1996.

\bibitem[Scheutzow(2013)]{Scheutzow2013}
Michael Scheutzow.
\newblock A stochastic {G}ronwall lemma.
\newblock \emph{Infinite Dimensional Analysis, Quantum Probability and Related
  Topics}, 16\penalty0 (02):\penalty0 1350019, 2013.

\bibitem[Scheutzow and Schulze(2017)]{ScheutzowSchulze2017}
Michael Scheutzow and Susanne Schulze.
\newblock Strong completeness and semi-flows for stochastic differential
  equations with monotone drift.
\newblock \emph{Journal of Mathematical Analysis and Applications},
  446\penalty0 (2):\penalty0 1555--1570, 2017.

\bibitem[von Weizs{\"a}cker and Winkler(1990)]{WeizaeckerWinkler1990}
Heinrich von Weizs{\"a}cker and Gerhard Winkler.
\newblock Stochastic integrals: An introduction, advanced lectures in
  mathematics, 1990.

\bibitem[Wang and Gan(2013)]{WangGan2013}
Xiaojie Wang and Siqing Gan.
\newblock {The tamed Milstein method for commutative stochastic differential
  equations with non-globally Lipschitz continuous coefficients}.
\newblock \emph{Journal of Difference Equations and Applications}, 19\penalty0
  (3):\penalty0 466--490, 2013.

\bibitem[Zhan et~al.(2018)Zhan, Jiang, and Liu]{ZhanJiangLiu2018}
Weijun Zhan, Yanan Jiang, and Wei Liu.
\newblock {A note on convergence and stability of the truncated Milstein method
  for stochastic differential equations}.
\newblock \emph{arXiv preprint arXiv:1809.05993}, 2018.

\bibitem[Zhang(2010)]{Zhang2010}
Xicheng Zhang.
\newblock Stochastic flows and {B}ismut formulas for stochastic {H}amiltonian
  systems.
\newblock \emph{Stochastic Process. Appl.}, 120\penalty0 (10):\penalty0
  1929--1949, 2010.

\bibitem[Zhou et~al.(2016)Zhou, Zhang, Hong, and Song]{ZhouEtAl2016}
Weien Zhou, Liying Zhang, Jialin Hong, and Songhe Song.
\newblock {Projection methods for stochastic differential equations with
  conserved quantities}.
\newblock \emph{BIT Numerical Mathematics}, 56\penalty0 (4):\penalty0
  1497--1518, 2016.

\bibitem[Zong et~al.(2014)Zong, Wu, and Huang]{ZongWuHuang2014}
Xiaofeng Zong, Fuke Wu, and Chengming Huang.
\newblock {Convergence and stability of the semi-tamed Euler scheme for
  stochastic differential equations with non-Lipschitz continuous
  coefficients}.
\newblock \emph{Applied Mathematics and Computation}, 228:\penalty0 240--250,
  2014.

\end{thebibliography}

\end{document}